\documentclass[11pt]{article}
\usepackage{amssymb,amsmath,amsthm,amsfonts,verbatim, mathabx}
\usepackage[all,cmtip]{xy}
\input diagxy
\usepackage{hyperref}
\usepackage[top=1.2in,bottom=1.2in,left=1.21in,right=1.21in]{geometry}
\usepackage{todonotes}
\usepackage{qtree}
\usepackage{amscd}
\usepackage{pdfsync}

\newtheorem{theorem}{Theorem}[section]
\newtheorem{proposition}[theorem]{Proposition}
\newtheorem{corollary}[theorem]{Corollary}
\newtheorem{lemma}[theorem]{Lemma}

\theoremstyle{definition}

\newtheorem{definition}[theorem]{Definition}
\newtheorem{convention}[theorem]{Convention}
\newtheorem{example}[theorem]{Example}
\newtheorem{remark}[theorem]{Remark}

\newtheorem{claim}{Claim}

\newtheorem{problem}[theorem]{Problem}

\numberwithin{equation}{section}

\newcommand{\into}{\hookrightarrow}

\newcommand{\Poly}{{\cal Z}}
\newcommand{\bd}{\vec{\mathbf{d}}}

\newcommand{\Ac}{\mathcal{A}}

\newcommand{\Mc}{\mathcal{M}}
\newcommand{\Oc}{\mathcal{O}}

\newcommand{\iso}{\cong}

\newcommand{\Ab}{\mathbb{A}}
\newcommand{\A}{\mathbb{A}}
\newcommand{\Cb}{\mathbb{C}}
\newcommand{\C}{\mathbb{C}}

\newcommand{\F}{\mathbb{F}}

\newcommand{\Nb}{\mathbb{N}}

\newcommand{\Qb}{\mathbb{Q}}
\newcommand{\Q}{\mathbb{Q}}
\newcommand{\Rb}{\mathbb{R}}

\newcommand{\Zb}{\mathbb{Z}}
\newcommand{\Z}{\mathbb{Z}}
\newcommand{\CP}{\mathbb{CP}}

\newcommand{\Gr}{Gr}

\DeclareMathOperator{\Sym}{Sym}

\DeclareMathOperator{\Hom}{Hom}

\DeclareMathOperator{\gr}{gr}

\DeclareMathOperator{\cd}{cd}
\DeclareMathOperator{\stab}{stab}

\DeclareMathOperator{\Prob}{Prob}

\DeclareMathOperator{\GL}{GL}
\DeclareMathOperator{\HD}{HD}

\DeclareMathOperator{\coor}{coor}

\renewcommand{\O}{\mathcal{O}}

\newcommand{\ra}{\rightarrow}
\newcommand{\tensor}{\otimes}

\newcommand{\para}[1]{\medskip\noindent\textbf{#1.}}

\title{Coincidences between homological densities, \\ predicted by arithmetic}
\author{Benson Farb,  Jesse Wolfson and Melanie Matchett Wood}

\begin{document}

\maketitle
\begin{abstract}
    Motivated by analogies with basic density theorems in analytic number theory, we introduce
    a notion (and variations) of the {\em homological density} of one space in another.  We use Weil's number field/ function field analogy to predict coincidences for limiting homological densities of various sequences $\Poly^{(d_1,\ldots ,d_m)}_n(X)$ of spaces of $0$-cycles on manifolds $X$.  The main theorem in this paper is that these topological predictions, which seem strange from a purely topological viewpoint, are indeed true.

    One obstacle to proving such a theorem is the combinatorial complexity of all possible ``collisions'' of points.  This problem does not arise in the simplest (and classical) case $(m,n)=(1,2)$ of configuration spaces.  To overcome this obstacle we apply the Bj\"orner--Wachs theory of lexicographic shellability from algebraic combinatorics.
   \end{abstract}

\section{Introduction}
The purpose of this paper is to introduce the notion of the ``homological density'' of one space in another, and to prove coincidences for limiting densities for various sequences of spaces of $0$-cycles on manifolds.  We were led to such coincidences by analogy with classical density results in analytic number theory. We do not yet understand why  these topological predictions end up being true.

\para{Spaces of 0-cycles} Let $X$ be a connected, oriented, smooth manifold with $\dim H^*(X;\Q)<\infty$ (this is a standing assumption throughout this paper).  Fix $m,n\geq 1$.  Let $\bd$ denote a tuple of non-negative integers  $(d_1,\ldots,d_m)\in\Zb_{\geq0}^m$, and let $|\bd|:=\sum_i d_i$.  Let $\Sym^d(X):=X^d/S_d$ be the $d^{th}$ symmetric product of $X$;   more generally, let $\Sym^{\bd}(X):=\prod_i \Sym^{d_i}(X)$.  Consider the space
$\Poly^{\bd}_n(X)\subset \Sym^{\bd}(X)$ of sets $D$ of $|\bd|$ (not necessarily distinct) points in $X$ such that:
\begin{enumerate}
\item precisely $d_i$ of the points in $D$ are labeled with the ``color'' $i$, and
\item no point of $X$ is labelled with at least $n$ labels of every color.
\end{enumerate}

\noindent
Such spaces of $0$-cycles include several basic examples in topology and geometry.
For example:
\begin{itemize}
\item $\Poly^d_2(X)$ is the configuration space of unordered $d$-tuples of distinct points in $X$.
\item $\Poly^{\overbrace{(d,\ldots,d)}^m}_1(\C)$ is the space of degree $d$, based rational maps $f:\CP^1\to \CP^{m-1}$ with $f(\infty)=[1:\cdots:1]$.
\end{itemize}
The space $\Poly^{\bd}_n(X)$ is a topological analogue of the set of ``relatively $n$-prime'' $m$-tuples of ideals in the ring of integers in a number field $K$.

\para{Homological densities}
The density of one set in another (e.g. square-free integers in the interval $[1,d]$) is a basic concept in analytic number theory.  Motivated by the framework of the Weil conjectures, we propose the following notion of ``homological density''.

Recall that the {\em Poincar\'{e} polynomial} $P_X(t)\in\Z[t]$ of a space $X$ with
finite-dimensional rational cohomology is defined by
\[P_X(t):=\sum_{i\geq 0} \dim_\Qb H^i(X;\Qb)t^i.\]

\begin{definition}
    Let $Y\subset Z$ be spaces with finite-dimensional rational cohomology. We define the {\em homological density} of $Y$ in $Z$ to be the ratio of Poincar\`e polynomials $\frac{\displaystyle P_Y(t)}{\displaystyle P_Z(t)}$.
\end{definition}

\para{Coincidences of limiting densities}
Results going back to the 19th century imply that the limiting density of the set of relatively $n$-prime $m$-tuples of ideals  in a ring of integers $\Oc_K$, considered within the set of all $m$-tuples of ideals, converges to $\zeta_K(mn)^{-1}$ (see \S\ref{S:NT} below); in particular, the limiting density only depends on the product $mn$ and the number field $K$. By considering this in connection with topological results of Arnol'd \cite{Ar70b}, Segal \cite{Se}, Cohen--Cohen--Mann--Milgram \cite{CCMM}, Vassiliev \cite{Va} and others, we were led to predict that analogous coincidences should hold for limiting homological densities for spaces of $0$-cycles. That these predictions are true is the main result of our paper.

\begin{theorem}[{\bf Coincidences between limiting homological densities}]
\label{theorem:main}

Let $X$ be a connected orientable smooth manifold with $\dim H^*(X;\Q)<\infty$ such that the cup-product of any $k$ compactly supported cohomology classes is $0$.
If $m_1,n_1,m_2,n_2$ are positive integers with $m_1n_1=m_2n_2\geq k$ then
\begin{equation}\label{eq:limit1}
    \lim_{\bd\in \mathbb{N}^{m_1} \rightarrow\infty}
    \frac{\displaystyle P_{\Poly^{\bd}_{n_1}(X)}(t)}{\displaystyle P_{\Sym^{\bd}(X)}(t)} =    \lim_{\bd\in \mathbb{N}^{m_2} \rightarrow\infty}
    \frac{\displaystyle P_{\Poly^{\bd}_{n_2}(X)}(t)}{\displaystyle P_{\Sym^{\bd}(X)}(t)},
\end{equation}
where the limits are in $t$-adic topology on the ring $\Z[[t]]$ of formal power series; in particular these limits exist as all $d_i\rightarrow\infty$, at any rates.
\end{theorem}

In Remark~\ref{remark:why:vanishing} we verify that the hypotheses of Theorem~\ref{theorem:main} hold in the following examples.

\begin{corollary}\label{C:main}
Let $X$ be a connected manifold with $\dim H^*(X;\Q)<\infty$.  The conclusion of Theorem~\ref{theorem:main} holds when: 
\begin{enumerate}
    \item $X$  a smooth affine variety over $\C$ and $k=3$; or
    \item $X$  an open submanifold of $\C^r$ and $k=2$; or
    \item \label{poincare3} $X$  an orientable, noncompact, smooth manifold and $k>\dim X$.
\end{enumerate}
\end{corollary}

    For each fixed $N>0$, Theorem \ref{theorem:main}  asserts that the limiting homological densities are the same for all spaces with $mn = N$.  This gives many coincidences of homological densities, one for each divisor of $N$.  For example when $mn=2$, one sees that the following different kinds of configuration spaces have the same limiting homological densities: 1)the space of indistinguishable particles on X with no two in the same location, and 2) the space of red and blue (otherwise indistinguishable) particles on X where no location has both a red and blue particle.
The main point of this paper is that these coincidences in topology exist, and that one needs to consider homological density to see them.
To prove Theorem~\ref{theorem:main}, we in fact compute the Poincare polynomials explicitly for every $X,\bd,m,n$; see \eqref{E:formula} and Theorem~\ref{C:formula}.


For odd-dimensional manifolds we prove a stronger statement. While extending the coincidences of Theorem \ref{theorem:main}, it also makes clear that, for odd-dimensional manifolds, the limiting homological densities are less interesting.
\begin{theorem}\label{theorem:odd}
    Let $X$ be a connected, oriented, smooth, manifold of dimension $2r+1\ge 3$ with $\dim H^*(X;\Q)<\infty$. Then the inclusion $\Poly^{\bd}_n(X)\into \Sym^{\bd}(X)$ induces an isomorphism on rational cohomology.  In particular the limit \eqref{eq:limit1} exists and equals $1$.
\end{theorem}
Theorem \ref{theorem:odd} appears as Statement \ref{cohoodd} of
Theorem \ref{theorem:coho} below.
Theorem \ref{theorem:odd} extends previous results of Felix--Tanr\`e\footnote{See also Bodigheimer--Cohen--Taylor \cite{BCT}.} \cite[Theorem 4]{FTa} who proved
Theorem \ref{theorem:odd} for configuration spaces, i.e. the case $(m,n)=(1,2)$.

The following is an illustration in a simple case of the content of Theorem~\ref{theorem:main}.

\begin{example}
\label{example:xminus0}
Consider the case when $X=\C^\times$ and $mn=2$.  A variation on theorems of Arnol'd and Segal (using Theorem~\ref{theorem:coho} below) gives:
\begin{equation}
\label{eq:poincare1}
\lim_{d\rightarrow\infty}P_{\Poly^d_2(\C^\times)}(t)=1+2t+2t^2+2t^3+\cdots
\end{equation}
and
\begin{equation}
\label{eq:poincare2}
\lim_{d\rightarrow\infty}P_{\Poly^{(d,d)}_1(\C^\times)}(t)=1+3t+4t^2+4t^3+\cdots
\end{equation}

An easy computation gives $P_{\Sym^d(\C^\times)}(t)=1+t$ for $d\geq 2$.  So while \eqref{eq:poincare1} and \eqref{eq:poincare2} are not equal, we find  that:

\[\lim_{d\rightarrow\infty}\frac{\displaystyle P_{\Poly^{d}_2(\C^\times)}(t)}{\displaystyle P_{\Sym^{d}(\C^\times)}(t)}=\frac{\displaystyle 1+2t+2t^2+2t^3\cdots}{\displaystyle 1+t}=1+t+t^2+t^3+\cdots\]
and
\[\lim_{d\rightarrow\infty}\frac{\displaystyle P_{\Poly^{(d,d)}_1(\C^\times)}(t)}{\displaystyle (P_{\Sym^d(\C^\times)}(t))^2}=\frac{\displaystyle 1+3t+4t^2+4t^3\cdots}{\displaystyle (1+t)^2}=1+t+t^2+t^3+\cdots\]
are equal.  This illustrates why one must take a quotient in Theorem~\ref{theorem:main}.  We remark that $\Poly^{(d,d)}_1(X)$ in this example can be replaced by $\Poly^{(d_1,d_2)}_1(X)$ for any $(d_1,d_2)\to\infty$.
\end{example}

\begin{remark}\label{R:CE}\mbox{}
The assumption that $mn>2$ when $X$ is an affine variety is sharp (and similarly in  \eqref{poincare3} of Corollary~\ref{C:main}): let $X=T^2-*$ be the punctured torus. For $mn=2$, one can compute using Theorem~\ref{theorem:coho} below that:
            \begin{align*}
                \lim_{d\to\infty}\frac{\displaystyle P_{\Poly^d_2(T^2-*)}(t)}{\displaystyle P_{\Sym^d(T^2-*)}(t)} &= 1+3t^2-t^3+\cdots\intertext{while}
                \lim_{d\to\infty}\frac{\displaystyle P_{\Poly^{d,d}_1(T^2-*)}(t)}{\displaystyle (P_{\Sym^d(T^2-*)}(t))^2}&=1+t^2+5t^3+\cdots
            \end{align*}
\end{remark}

\para{Homological Stability}
We deduce the existence of the limits in Theorem \ref{theorem:main} from the following.  For an $m$-tuple $\bd=(d_1,\ldots,d_m)$, define $\bd+1_i:=(d_1,\ldots,d_i+1,\ldots,d_m)$.

\begin{theorem}[{\bf Rational homological stability for spaces of 0-cycles}]\label{theorem:stable}
    Let $X$ be a smooth, orientable, connected manifold with $\dim(X)\geq 2$. For each $1\leq i\leq m$, there exists a natural (in $X$) map
    \begin{equation*}
        H^\ast(\Poly^{\bd}_n(X);\Qb)\to H^\ast(\Poly^{\bd+1_i}_n(X);\Qb)
    \end{equation*}
    that is an isomorphism for $*\leq d_i$ when either $r\geq 2$, $m\geq 2$ or $n\geq 3$, and $*\leq d_i/2$ when $(r,m,n)=(1,1,2)$, this last case being the case of configuration spaces on surfaces.
\end{theorem}

In light of Theorem \ref{theorem:odd}, the content of Theorem \ref{theorem:stable} is for $\dim(X)$ even.

\begin{remark}\mbox{}

\begin{enumerate}
\item For $X=\Cb^n$, the special case $\bd=(d,\ldots,d)$ of Theorem \ref{theorem:stable} is a consequence of Theorem 1.2(2) of \cite{FW}.  For $X=\Cb^n$ and general $\bd$, Theorem \ref{theorem:stable} was first proven by Gadish \cite[Theorem 6.13]{Ga} as a special case of his theory of finite generation for families of linear subspace arrangements.

\item For general $X$, Theorem \ref{theorem:stable} gives a simultaneous generalization of stability for configuration spaces  \cite{Ch},\cite{RW}, \cite{BM}, \cite{Kn} (the $(m,n)=(1,2)$ case, proved only recently), for bounded symmetric powers  \cite{Va,KM} (the $(m,n)=(1,n),n>2$ case), and for spaces of rational maps \cite{Se} (the $(m,n)=(m,1),m\geq 2$ case).

\item Arabia \cite{Arabia} has considered a different generalization of (ordered) configuration spaces than the kind we consider here, and defined a notion of \emph{i-acyclic}, which is equivalent to requiring that the cup product of any $2$ compactly supported cohomology classes is $0$.
For \emph{i-acyclic} spaces, he gives formulas for the Poincar\'{e} polynomials of his generalized ordered configuration spaces in terms of the Poincar\'{e} polynomial of $X$, in the same spirit as the formulas we give in \eqref{E:formula} and Theorem~\ref{C:formula}.

\item The main difference between Theorem \ref{theorem:stable} and these previously studied cases is the greater complexity of the allowed collisions of the particles.  Regardless of which topological method one uses to compute cohomology (we use the classical Leray spectral sequence), one needs to keep track of this combinatorics, and to interface this information with the topological tools used.  This is the bulk of the proof of Theorem \ref{theorem:main}.  A crucial ingredient is the the Bj\"{o}rner-Wachs theory of lexicographic shellability.

\item The shellability point-of-view shows why the case of configuration spaces is particularly simple: the associated partition lattice is a pure poset.

\end{enumerate}
\end{remark}

\para{Two variants}
We prove Theorem \ref{theorem:main} as an application of an explicit computation (Theorem~\ref{theorem:coho}) of the $E_2$-page and first non-trivial differential of the Leray spectral sequence for the sheaf $\Qb$ and the inclusion $\Poly^{\bd}_n(X)\to\Sym^{\bd}(X)$. The failure of this differential to vanish is responsible for the failure of the coincidence in the example $T^2-*$ and $mn=2$ above. We now state two variants of Theorem \ref{theorem:main} that bypass this differential.

Note that $P_X(-1)=\chi(X)$ for any space $X$. Further, recall that when $X$ is also an algebraic variety, it comes equipped with a mixed Hodge structure, giving {\em Hodge-Deligne numbers} $h^{p,q}(X)\geq 0$ (see \S\ref{section:mixedhodge} below).  These numbers can be concisely packaged into the {\em Hodge-Deligne polynomial} \footnote{This is slightly nonstandard usage: the Hodge-Deligne polynomial is more commonly defined using compactly supported cohomology.} :
\[\HD_X(u,v):=\sum_{p,q\geq 0}h^{p,q}(X)u^pv^q \in \Z[u,v]\]
 and if $X$ is a smooth projective variety then $\HD_X(t,t)=P_X(t)$.

\begin{theorem}
\label{theorem:second}
Fix $m,n\geq 1$  and let $\bd=(d_1,\ldots ,d_m)$, and let $\lim_{\bd\rightarrow\infty}$  mean ``as all $d_i\rightarrow\infty$'', at any rates.
\begin{enumerate}
    \item For $X$ a connected, oriented, smooth, even-dimensional manifold with $\dim H^*(X;\Q)<\infty$,\footnote{For example, $X$ is compact or is the interior of a compact manifold with boundary.}
        \[
            \frac{ \sum_{\bd\in \Zb_{\geq0}^m} \chi(\Poly^{\bd}_n(X)) x^{|\bd|}   }{
            \sum_{\bd\in \Zb_{\geq0}^m} \chi(\Sym^{\bd}(X)) x^{|\bd|}
            }=(1-x^{mn})^{\chi(X)}.
        \]
    In particular, this only depends on the product $mn$ and on $\chi(X)$.
\item For $X$ a connected, smooth complex-algebraic variety, the limit
   \begin{equation}\label{E:HDlimit}
    \lim_{\bd\rightarrow\infty}\frac{\HD_{\Poly^{\bd}_n(X)}(u,v)}{\HD_{\Sym^{\bd}(X)}(u,v)}
\end{equation}
   exists in the adic topology on $\Z[[u,v]]$, and depends only on the product $mn$, the mixed Hodge structure on $H^*(X;\Q)$, and $\dim X$.
   \end{enumerate}
\end{theorem}

Theorem \ref{theorem:second} avoids the assumptions of Theorem \ref{theorem:main} because the Euler characteristic and the Hodge--Deligne polynomial do not distinguish between the $E_2$ and $E_\infty$-pages in a spectral sequence. It would be interesting to extract appropriate ``correction terms'' from the differentials so that the (corrected) limiting homological densities coincide in general.
Getzler \cite{Getzler} has given generating function formulas for the Hodge--Deligne polynomials of ordered configuration spaces, and as a corollary one has that the limit \eqref{E:HDlimit} in the case $m=1$ is as predicted by the analogy with arithmetic.
The second part of Theorem \ref{theorem:second} would also follow from an analogous statement in the Grothendieck ring of varieties.   We conjecture
the limit \eqref{E:HDlimit} lifted to the  Grothendieck ring of varieties is $\zeta_X([\A^1]^{-mn})^{-1}$, where  $\zeta_X$ is Kapranov's motivic zeta function.
The $m=1$ case of this conjecture, as well as similar results for other generalizations of configuration spaces is proven in \cite[Theorem 1.30]{VakilWood}.

\para{How arithmetic predicts the coincidences in Theorem~\ref{theorem:main}}  We originally conjectured the coincidences of Theorem~\ref{theorem:main} by analogy with density results in arithmetic.  As a simple example, consider the following heuristic for the density of the set of square-free integers among the set of all integers:
\begin{flalign*}
\lim_{d\rightarrow\infty} \frac{\#\{n\in [1,d]:n\neq p^2\}}{\#\{n\in[1,d]\}}& =\lim_{d\rightarrow\infty}\prod_{p \text{\ prime}\le d}\Prob(p^2\nmid n)\\
&=\prod_{p \text{\ prime}}(1-\frac{1}{p^2})=\zeta(2)^{-1}
\end{flalign*}
where $\zeta(s):=\sum_{n=1}^\infty\frac{1}{n^s}$ is the Riemann zeta function.   A heuristic for the density of pairs of relatively prime integers among all pairs of integers is given by:
\begin{flalign*}
\lim_{d\rightarrow\infty} \frac{\#\{(m,n)\in [1,d]^2:\gcd(m,n)=1\}}{\#\{(m,n)\in[1,d]^2\}}&=\lim_{d\rightarrow\infty}\prod_{p \text{\ prime}\le d}[1-\Prob(p|m\text{\ and\ }p|n)]\\
& =\prod_{p \text{\ prime}}(1-\frac{1}{p^2})=\zeta(2)^{-1}
\end{flalign*}

Both heuristics above are accurate: it has been known since the 1800s \cite{Gegenbauer, Mertens} that each density is indeed $\zeta(2)^{-1}$.  This is a well-known coincidence.  Note that in the two limits we divide by $d$ and $d^2$, respectively, corresponding to the cardinality of the ``background spaces''
$[1,d]$ and $[1,d]^2$.

Weil espoused a powerful analogy between number fields and function fields (over $\Cb$ and over finite fields).  This analogy gives in particular the following correspondences:

\begin{center}
\begin{tabular}{|c|c|}
\hline
Number field \qquad\quad& Function field \\
\hline
$[1,d]$ & ${\rm Pol}_d:=\{\text{monic\ } f\in\C[t]: \deg(f)=d\} \iso \Sym^d(\Cb)$\\
\hline
$\{\text{square-free\ }n\in [1,d]\}$
& $\{\text{square-free\ }f\in{\rm Pol}_d\}\iso \Poly^d_2(\Cb)$\\
\hline
$\{(m,n)\in [1,d]^2:\gcd(m,n)=1\}$& $\{(f_0,f_1)\in{\rm Pol}_d^2:\gcd(f_0,f_1)=1\}\iso \Poly^{d,d}_1(\Cb)$\\
\hline
\# & (co)homology\\
\hline
\end{tabular}
\end{center}

Taking this analogy seriously, from the two examples above one might guess two things: first, that $\lim_{d\rightarrow\infty}H_\ast(\Poly^d_2(\C);\Z)$ and $\lim_{d\rightarrow\infty}H_\ast(\Poly^{d,d}_1(\C);\Z)$ exist; and second, that these limits are equal.  This is true: in two highly influential papers, Arnol'd \cite{Ar} and Segal \cite{Se} proved existence of these limits, and showed that they each equal $H_\ast(\Omega_0^2\CP^1;\Z)$, the basepoint component of the second loop space of the 2-sphere.  \footnote{There are many other such coincidences.  For example, generalizing the above, for any fixed $n\geq 2$ the density of ``$n$-power-free integers'' among all integers equals $\zeta(n)^{-1}$, which also happens to be the density of $n$-tuples of integers with common gcd $1$ among all $n$-tuples of integers. See Section~\ref{S:NT} for more details on the analogy in number theory. The topological analogs were proved by Vassiliev \cite{Va} and Segal \cite{Se}, with limiting homology that of $\Omega_0^2\CP^{n-1}$.}

One might try to push the analogy further, replacing $\C$ with other open manifolds $X$. However, as we see from Example~\ref{example:xminus0} for $X=\C^\times$:
\[\lim_{d\rightarrow\infty}H_i(\Poly^d_2(\C^\times);\Q)\neq \lim_{d\rightarrow\infty}H_i(\Poly^{d,d}_1(\C^\times);\Q).\]

What went wrong? The answer lies in the fact that we didn't take Weil's analogy seriously enough: we need to somehow ``divide'' by the spaces corresponding to $[1,d]$ and $[1,d]^2$, namely $\Sym^d(X)$ and $\Sym^{d}(X)^2$.  As indicated by Theorem~\ref{theorem:main}, interpreting this division as division of Poincar\'{e} polynomials gives a correct theorem in many examples.  Note that the necessity of dividing was not visible in the example when $X=\C$ since, by Newton's Theorem, $\Sym^d(\C)\iso \C^d$, and so $P_{\Sym^d(\C)}(t)=1=(P_{\Sym^d(\C)}(t))^2$.

One can ask what exactly about cohomology should arise in the analogy above.    The idea that the function field analog of counting is an Euler characteristic or Hodge-Deligne polynomial is
suggested by the Grothendieck-Lefschetz trace formula.  This is a well-understood analogy, and via this analogy Theorem~\ref{theorem:second} is predicted.
In \cite{VakilWood}, based on theorems about Hodge-Deligne polynomials motivated by arithmetic,  Vakil and the third author posed many questions about about actual Betti numbers, asking how far this analogy might extend to topology and to what extent it can predict not just Euler characteristics but Betti numbers.  One point of this paper is that these analogies from arithmetic can be extended to topology beyond just Euler characteristics, as seen in Theorem~\ref{theorem:main}, but this extension is more mysterious than the well-understood analogy with Euler characteristics or Hodge-Deligne polynomials, as seen by Remark~\ref{R:CE}.

\para{Outline of the proof of Theorem \ref{theorem:main}}
\label{outline}
We deduce Theorems \ref{theorem:main} and \ref{theorem:odd} from an explicit description of the $E_2$-page of the Leray spectral sequence for the inclusion  $\Poly^{\bd}_n(X)\subset \Sym^{\bd}(X)$; this description is the content of Theorem \ref{theorem:coho}. The proof of Theorem \ref{theorem:coho} is quite involved and  takes up Sections~\ref{section:nequals}-\ref{section:final}.  In outline,  the proof of Theorem \ref{theorem:coho} proceeds as follows.

\begin{enumerate}
\item We start by considering an ordered version $\widetilde{\Poly}^{D}_n(X)$ of $\Poly^{\bd}_n(X)$, defined for any ``$m$-colored'' set $D$, on which a product of symmetric groups $S_D\cong S_{d_1}\times\cdots S_{d_m}$ acts with quotient $\Poly^{\bd}_n(X)$.    Our first goal is to analyze $H^*(\widetilde{\Poly}^{D}_n(X);\Q)$ by using the Leray spectral sequence for the inclusion  $\pi: \widetilde{\Poly}^D_n(X)\ra X^D\cong X^{|\bd|}$.

\item The $E_2$-page is given by $H^p(X^D;R^q\pi_\ast\Zb)$.  We must therefore understand the coefficient sheaves $R^q\pi_\ast\Zb$. In \S\ref{section:0cycles}, we reduce this, using the Goresky-MacPherson formula, to a combinatorial problem expressed in terms of the homology of order complexes associated to certain posets of so-called ``colored $n$-equals partitions''. In contrast to configuration spaces (the case $(m,n)=(1,2)$), the possibility of  particle collisions leads to much greater combinatorial complexity of the relevant partition lattices.

\item To handle this complexity, we make critical use of Bj\"orner--Wachs' theory of ``lexicographic shellability'' \cite{BW,BW2}.   This theory gives a method for proving that certain combinatorially-defined complexes are homotopy equivalent to a wedge of spheres, and we show in Section \ref{section:nequals} that it holds for the order complexes of posets of colored $n$-equals partitions. We first use this to give a {\em qualitative} description of the $E_2$-page of the ordered case in terms of sheaves supported on diagonals in $X^{|\bd|}$ (Theorem \ref{theorem:leray}).

\item The next ingredient for the computation of $H^*(\Poly^{\bd}_n(X);\Q)$ is the ``local case'' $X=\Rb^N$.  In \S\ref{section:local} we give  (Theorem \ref{theorem:local}) an explicit computation of $H^\ast(\Poly^{\bd}_n(\Rb^N);\Q)$ .  The key idea is to consider a filtration of the ``discriminant locus'' in $\Sym^{\bd}(\Rb^N)$, whose complement is $\Poly^{\bd}_n(\Rb^N)$,
    and then to use associated cofiber sequences in an inductive argument.  This is similar to the work of Farb-Wolfson \cite{FW}, which in turn built on work of Segal \cite{Se} and Arnol'd \cite{Ar70b}.

\item In \S\ref{section:final} we combine the local computation with the combinatorial results of Section \ref{section:nequals} to obtain a {\em quantitative} description of the $S_D$-invariants of the $E_2$-page; by transfer this gives Theorem \ref{theorem:coho}.

\item In Section \ref{S:ss}, we use the description of the $E_2$-page to obtain information about $H^\ast(\Poly^{\bd}_n(X);\Q)$ and prove Theorem \ref{theorem:main}.
\end{enumerate}

\section{Analogies in number theory}
\label{S:NT}

In this section we indicate the statements in number theory that led us to the statement
of Theorem \ref{theorem:main}.

Given integers $a_1,\dots, a_m$, we say they are {\em relatively $n$-prime} if there does not exist an integer $b\geq 2$ such that $b^n\mid a_i$ for all $i$; in other words, if $\gcd(a_1,...,a_m)$ is $n$-power-free.  Let $\zeta(s)$ be the Riemann zeta function.  The following is a standard result in number theory.

\begin{theorem}[see, e.g. \cite{Benkoski}]\label{T:NT}
Given positive integers $m$ and $n$, the limit
$$
\lim_{d\ra\infty}
\frac{\#\{(a_1,\dots,a_m)\in (\mathbb{N}_{\leq d} )^m | \textrm{$a_1,\dots,a_m$ relatively $n$-prime}  \}}{\#\{(a_1,\dots,a_m)\in (\mathbb{N}_{\leq d} )^m   \}}
$$
exists and equals $\zeta(mn)^{-1}$.  In particular, this limit only depends on the product $mn$.
\end{theorem}

Moreover, such a statement holds if we replace $\Z$ with the ring of integers $\O_K$ in an number field $K$,  the set $\mathbb{N}_{\leq d}$ with the set of ideals of $\O_K$ of norm at most $d$, the function $\zeta(s)$ with the Dedekind zeta function $\zeta_K(s)$, and relatively $n$-prime with the requirement that there be no non-trivial ideal $b\subset \O_K$ such that $b^n\mid a_i$ for all $i$.

In the usual analogy between number fields and function fields, we can also replace $\Z$ with the ring of integers in a function field over a finite field such as $\F_q[t]$. Let $S_d$ be the set of monic polynomials of degree $d$ in $\F_q[t]$.   A set of polynomials is {\em relatively $n$-prime} if there does not exist a non-constant polynomial $b$ that divides all of them.
\begin{theorem}[see, e.g. \cite{Morrison}]
Given positive integers $m$ and $n$, the limit
$$
\lim_{d\ra\infty}
\frac{\#\{(a_1,\dots,a_m)\in (S_d)^m | \textrm{$a_1,\dots,a_m$ relatively $n$-prime}  \}}{\#\{(a_1,\dots,a_m)\in (S_d )^m   \}}
$$
exists and equals $\zeta_{\F_q[t]}(mn)^{-1}$.  In particular, this limit only depends on the product $mn$.
\end{theorem}

Similarly to the above, the analogous version is also true when we replace $\F_q[t]$ with the ring of integers in any function field over any finite field.  The function field statements can also be interpreted geometrically as the following.

\begin{theorem}\label{T:FF}
Let $X$ be a  curve (not necessarily complete or smooth) over a finite field $\F_q$ with local zeta function $\zeta_X(s)$.
 Given positive integers $m$ and $n$, the limit
\begin{equation}
\lim_{\bd\rightarrow\infty}
\frac{\displaystyle {\#\Poly^{\bd}_n(X)}(\F_q)}{\displaystyle (\#{\Sym^{\bd}(X)}(\F_q))}
\end{equation}
exists and equals $\zeta_{X}(mn)^{-1}$.  In particular, the limit only depends on the product $mn$.
\end{theorem}

The number field version ``with punctures'' holds as well, taking the zeta function without the factors in the Euler product corresponding to the punctures.
In fact, Theorem~\ref{T:FF} holds for any connected variety $X$, not just one-dimensional $X$, with $\zeta_{X}(mn)^{-1}$ replaced by $\zeta_{X}(mn\dim X)^{-1}$.

\section{The Leray spectral sequence for $H^*(\Poly^{\bd}_n(X);\Q)$, and \\ applications}\label{S:ss}
In this section we state our main technical theorem, Theorem~\ref{theorem:coho} below, which gives the $E_2$ page of a spectral sequence converging to the cohomology of $\Poly^{\bd}_n(X)$.    We then apply this theorem as a black box to prove Theorems \ref{theorem:main}, \ref{theorem:odd}, \ref{theorem:stable} and \ref{theorem:second} given in the introduction.  The proof of Theorem~\ref{theorem:coho} will then take up the rest of the paper.

\subsection{Statement of the main technical theorem}

We want to understand the cohomology of the space $\Poly^{\bd}_n(X)$ for $X$ a smooth manifold with $\dim H^*(X;\Q)<\infty$.  To this end, we consider the Leray spectral sequence for  the inclusion $\Poly^{\bd}_n(X)\subset \Sym^{\bd}(X)$ and the constant sheaf $\Q$.  We denote the $(p,q)$ term of the $j^{th}$ page of this spectral sequence by $E^{p,q}_j(X,\bd,n)$.    Note that $E^{\ast,\ast}_j(X,\bd,n)$ is a bigraded algebra, graded by $(p,q)$.

We will need some notation for certain bigraded vector spaces. Denote by $\Qb[i]$ the rank 1 vector space of bidegree $(0,i)$, and by $H^j(X;\Qb[i])$ the vector space $H^j(X;\Q)$ with bidegree $(j,i)$.
Given any bigraded vector space $V$, the symmetric group $S_k$ acts on $V^{\tensor k}$ in the graded way with respect to the total grading, i.e. as in the K\"unneth formula.  Let $\Sym^{k}_{gr} V$ denote the trivial $S_k$ subpresentation of $V^{\tensor k}$.

The following theorem is the main technical result of this paper. For the definitions from mixed Hodge theory necessary to understand part (2) of the theorem, see \S\ref{section:mixedhodge} below.

\begin{theorem}[{\bf Cohomology of spaces of $0$-cycles}]
\label{theorem:coho}
    Let $X$ be a connected, smooth, orientable manifold. Fix $\bd\in\Nb^m$ and $n>0$.
    \begin{enumerate}
    \item \label{cohoodd} If $\dim(X)=2r+1, r>0,$ then the inclusion $\Poly^{\bd}_n(X)\into\Sym^{\bd}(X)$ induces an isomorphism on rational cohomology.
        \item \label{cohoodd2} If $\dim(X)=2r, r>0,$ then $E_2^{p,q}(X,\bd,n)=0$ unless $p+q\leq 2r|\bd|$ and $q/(2r(mn-1)-1)\in \Nb_{\le \min_i \frac{d_i}{n}}$, in which case it is isomorphic to the degree $(p,q)$ part of :
            \begin{equation}\label{eq:lerayinv}
                \Sym_{gr}^{q/(2r(mn-1)-1)} H^*(X;\Qb[2r(mn-1)-1])\otimes\bigotimes_{i=1}^m \Sym_{gr}^{d_i-nq/(2r(mn-1)-1)} H^*(X; \Qb[0])
            \end{equation}
where $H^p(X;\Qb[q])$ has bidegree $(p,q)$, and where bidegrees are additive under symmetric powers.
        \item \label{cohoodd3} If $X$ is a smooth complex variety with $\dim_\C(X)=r>0$ then $E_2^{p,q}(X,\bd,n)$ is isomorphic, \emph{with its mixed Hodge structure} to the degree $(p,q)$ part of:
            \begin{align*}
                &\Sym_{gr}^{q/(2r(mn-1)-1)} H^*(X;\Qb[2r(mn-1)-1](r(mn-1),r(mn-1)))\\
                &\otimes\bigotimes_{i=1}^m \Sym_{gr}^{d_i-nq/(2r(mn-1)-1)} H^*(X; \Qb[0](0,0)).
            \end{align*}
            where $\Qb[i](c,c)$ denotes the rank 1 vector space of bidegree $(0,i)$ and pure Hodge structure of weight $2c$.
    \end{enumerate}
\end{theorem}

\begin{remark}
Theorem~\ref{theorem:coho} only gives the additive structure of $H^*(\Poly^{\bd}_n(X),\Q)$.  The precise multiplicative structure seems more subtle. \footnote{This question has been addressed in a recent work of Ho \cite{Ho}.}  However, after passing to an associated graded of a certain filtration,  $H^*(\Poly^{\bd}_n(X),\Q)$ has a particularly nice multiplicative structure which we hope to address in forthcoming work.
\end{remark}

\subsection{Application 0: Homological stability}

In this subsection we deduce Theorem~\ref{theorem:stable} from Theorem~\ref{theorem:coho}.   When $\dim(X)$ is odd, the theorem follows from Statement \ref{cohoodd} of Theorem~\ref{theorem:coho} together with classical rational homological stability for symmetric powers.  We thus assume that $\dim(X)$ is even.

Given $\bd$, $n$ and $X$, consider the $S_{\bd}$-equivariant map between the spaces of ordered $0$-cycles $\widetilde{\Poly}^{\bd+1_i}_n(X)\to  \widetilde{\Poly}^{\bd}_n(X)$ given by forgetting the last point of color $i$.  This map induces an $S_{\bd}$-equivariant map between the associated Leray spectral sequences for the inclusion into the ordered products.    Transfer gives a corresponding map $\Psi: E_r^{p,q}(X,\bd,n)\to E_r^{p,q}(X,\bd+1_i,n)$ between the Leray spectral sequences computing $H^\ast(\Poly^{\bd}_n(X);\Qb)$ and $H^\ast(\Poly^{\bd+1_i}_n(X);\Qb)$.   The $E_2$ pages of these spectral sequences are given explicitly by Statement \ref{cohoodd2} of Theorem~\ref{theorem:coho}.  The proof of
Theorem~\ref{theorem:coho} will show that for $q\leq d_i\cdot (2r(mn-1)-1)/n$ and $p+q\leq 2r|\bd|$
the map $\Psi$ is given by the tensor product of the identity map on the first big factor of \eqref{eq:lerayinv}, the identity map on all but the $i^{\rm th}$ factor of the second big factor, and the inclusion $\Sym_{gr}^{d_i-nq/(2r(mn-1)-1)} H^*(X; \Qb[0])\to  \Sym_{gr}^{d_i+1-nq/(2r(mn-1)-1)} H^*(X; \Qb[0])$ on the remaining factor.  Applying classical rational homological stability for symmetric products gives the desired result and the stated stable range.

\subsection{Application 1: Coincidences between limiting homological densities}
 \label{section:odd}

In this subsection we apply Theorem~\ref{theorem:coho} to deduce Theorem~\ref{theorem:main}.
To understand the differential in our spectral sequence, we will recall some basic facts about the differential in the Leray spectral sequence for the complement of a closed submanifold.
\begin{lemma}\label{L:map}
Let $Y$ be a smooth manifold and let $Z$ be a smooth, closed submanifold with orientable normal bundle.  Let $k\geq 1$ be the codimension of $Z$ in $Y$.  There is a map
$$
H^*(Z;\Q) \ra H^{*+k}(Y;\Q)
$$
described in any of the following equivalent ways.
\begin{enumerate}
\item The differential $d_k: H^*(Z;\Q) \cong E_2^{*,k-1}   \ra E_2^{k+*,0}\cong H^{*+k}(Y;\Q)$ in the Leray spectral sequence for the inclusion $Y\setminus Z \ra Y$ with $\Q$ coefficients.
\item The composite $H^*(Z;\Q) \stackrel{\cup Th_{N_{Z/Y}}}{\longrightarrow} H^{*+k}(N_{Z/Y}, N_{Z/Y}\setminus 0;\Q)\cong
 H^{*+k}(Y, Y\setminus Z;\Q) \rightarrow H^{*+k}(Y;\Q)$ of the Thom isomorphism for the normal bundle $N_{Z/Y}$, an isomorphism from the tubular neighborhood theorem and excision, and the map from relative cohomology to cohomology.
\item The composite $H^*(Z;\Q) \cong \Hom(H_c^{\dim Z-*}(Z;\Q),\Q ) \ra \Hom(H_c^{\dim Z-*}(Y;\Q),\Q ) \cong
H^{*+k}(Y;\Q)$ of the Poincar\`e duality map for $Z$, the map from the usual pull-back of compactly supported cohomology and the
 the Poincar\`e duality map for $Y$.
\item When $Z$ is connected and orientable, and $Y=Z^\ell$ with $Z\subset Y$ the diagonal, and $e_i$ is a graded basis for $H_c^*(Z;\Q)$ and  $\check{e}_i$ is a Poincar\`e dual basis
for $H^*(Z;\Q)$
and $T:H^{\dim Z}(Z;\Q)\cong \Q$, the map
$$\alpha \mapsto \sum_{i_1,\dots,i_\ell } T(e_{i_1}\cup\cdots \cup e_{i_{\ell}}\cup \alpha) \check{e}_{i_1}\tensor \cdots \tensor \check{e}_{i_{\ell}}. $$
\end{enumerate}
\end{lemma}
\begin{proof}
The equivalence of 1 and 2 follows by a similar argument as the identification of the differential in \cite[p.177-178]{BottTu}.  The equivalence of 2 and 3 is explained in \cite[p.65-69]{BottTu}.  It is easy to work out 4 as an explicit version of 3; see for example Chapter 11 of \cite{MS}.
\end{proof}


\begin{theorem}\label{T:getvannew}
Let $X$ be a connected orientable smooth manifold such that the cup-product of any $mn$ compactly supported cohomology classes is $0$.  Then all the differentials of the Leray spectral sequence for the inclusion  $\Poly^{\bd}_n(X)\subset \Sym^{\bd}(X)$ vanish.
\end{theorem}

\begin{remark}
\label{remark:why:vanishing}
Note that when $X=\Rb^{r}$ the hypothesis of Theorem~\ref{T:getvannew} holds by \cite[Lemma 1.2.4]{Arabia}, and thus it similarly holds for any open submanifold of $X=\Rb^{r}$ by {\em loc. cit.}.  If $X$ is an affine variety over $\C$, then each non-zero $e_i\in H^*_c(X,\Q)$ has degree $\geq \frac{1}{2} \dim X$ (by the Andreotti-Frankel Theorem \cite{AF}), and $e_1\cup e_2\cup e_3$ is thus $0$.
As another example, if $X$ is a connected orientable non-compact smooth manifold,
 then since each non-zero $e_i\in H^*_c(X,\Q)$ has degree at least 1 (by $X$ non-compact) we have $e_1\cup \cdots \cup e_k=0$ for $k>\dim X$.
This proves Corollary~\ref{C:main}.
In general, if $X$ is a  connected orientable non-compact smooth manifold,
and $s$ is the highest degree in which $X$ has non-vanishing cohomology, then for $k>\frac{\dim X}{\dim X-s}$, the hypothesis of  Theorem~\ref{T:getvannew} holds by a similar argument.
\end{remark}

\begin{proof}
We have the complement of the diagonal $j : X^{mn}\setminus X \rightarrow X^{mn}$.
 We claim that the differential $d_{\dim X(mn-1)}$ vanishes on $E^{0,\dim X(mn-1)-1}$ for the Leray spectral sequence for $j$ with rational coefficients.  The differential is given by Lemma~\ref{L:map} (4) and is clearly seen to be $0$ if and only if    the cup-product of any $mn$ compactly supported cohomology classes is $0$.


 Define $\widetilde{\Poly}^D_n(X)\subseteq X^{D}$
 to be the space of tuples of (not necessarily distinct) points in $X$ labeled by the elements of $D$ such that no point of $X$ has at least $n$ labels of each color.
 For any choices of $n$ integers from $1$ to $d_i$ for each $i$ from $1$ to $m$, we have a morphism from  $\pi : \tilde{\Poly}^{\bd}_n(X) \rightarrow X^{|\bd|}$ to the inclusion $j$ that projects $X^{|\bd|}$ to the $mn$ chosen coordinates.

From Theorem~\ref{theorem:leray} and Lemma~\ref{L:BW}, we have that in the Leray spectral sequence for $\pi$, $$E_2^{0,\dim X(mn-1)-1}=\bigoplus_{\substack{I\in \Pi^D_n\\I \textrm{ singletons except one subset of size }mn}}
H^0(X_I,\Z)$$
and
$E_2^{\dim X(mn-1),1}=H^{\dim X(mn-1)}(X^{|\bd|};\Z).$ Moreover, we can see $d_{\dim X(mn-1)}$ is trivial on $H^0(X_I,\Z)$ here because it pulls back from the Leray differential for the inclusion $j$ via the choices of coordinates given by $I$.

Consider the subalgebra $A$ of the $E_2$ page generated by the bottom row and $E_2^{0,\dim X(mn-1)-1}$.  By the above
$d_{\dim X(mn-1)}$ vanishes on $A$.  All lower or higher differentials then also vanish on $A$ because they are forced by degree to vanish on the generators of $A$.  By Theorem~\ref{theorem:leray}, Lemma~\ref{L:SIinv}, and the proof of Theorem~\ref{theorem:coho} (specifically \eqref{equation:sfinvariance}), the algebra $A$ includes all the $S_D$ invariants of the $E_2$ page, and thus all  differentials vanish for the Leray spectral sequence for $\Poly^{\bd}_n(X)\subset \Sym^{\bd}(X)$, as desired.
\end{proof}

Thus under the hypothesis of Theorem~\ref{T:getvannew}, for with $mn\geq k$, we have, by Theorem~\ref{theorem:coho}
\begin{equation}\label{E:formula}
\sum_{\bd \in \mathbb{N}^{m}} [H^*(\Poly^{\bd}_{n}(X))] =
[\Sym_{gr}^* H^*(X,\Q)]
\prod_{i=1}^{m} [\Sym_{gr}^* H^*(X,\Q)]
\end{equation}
is an equality in the Grothendieck ring of $\Z^{(m+1)}$ graded vector spaces (replacing the bigrading above), where
$H^j(\Poly^{\bd}_{n}(X))$ has grade $(j,d_1,\dots,d_m)$ and the first $H^*(X,\Q)$ has grade
$(*+2\dim X(mn-1)-1, n,\dots,n)$ and the $i$th  $H^*(X,\Q)$ in the product has grade
$(*,0,\dots,0,1,0,\dots,0)$, where the $1$ is in the $i+1$ coordinate, and the grading used in the definition of
$\Sym_{gr}$ is given by the first coordinate.
Since $$
\sum_{\bd \in \mathbb{N}^{m}} [H^*(\Sym^{\bd}(X))] =
\prod_{i=1}^{m} [\Sym_{gr}^* H^*(X,\Q)],
$$
with the analogous gradings,
we can conclude the following result about the limits in Theorem~\ref{theorem:main}.

\begin{theorem}
\label{C:formula}
    Fix positive integers $m,n$ with $mn\geq 2$. Let $X$ be a connected, smooth orientable manifold with $\dim H^*(X;\Q)<\infty$.  Suppose that the Leray spectral sequence for the inclusion
    $\Poly^{\bd}_n(X)\subset \Sym^{\bd}(X)$ and the sheaf $\Qb$ degenerates on the $E_2$ page.  Let $b_i(X):=\dim H^i(X;\Q)$.
     Then
    \begin{align*}
        \lim_{\bd\rightarrow\infty}
        \displaystyle P_{\Poly^{\bd}_n(X)}(t)&=
        \prod_{i\geq 0} (1 - (-t)^{i+2r(mn-1)-1} )^{-(-1)^{i+2r(mn-1)-1} b_i(X)}
        \left(\prod_{i\geq 1} (1 - (-t)^i )^{-(-1)^i b_i(X)} \right)^m\intertext{and}
        \lim_{\bd\rightarrow\infty}
        \frac{\displaystyle P_{\Poly^{\bd}_n(X)}(t)}{\displaystyle (P_{\Sym^{\bd}(X)}(t))} &=
        \prod_{i\geq 0} (1 - (-t)^{i+2r(mn-1)-1} )^{-(-1)^{i+2r(mn-1)-1} b_i(X)}.
    \end{align*}

    Here $\lim_{\bd\rightarrow\infty}$ means ``as all $d_i\rightarrow\infty$'' (at any rates), and we take the limit in the ring of formal power series $\Zb[[t]]$ with the usual $t$-adic topology.
\end{theorem}

\begin{proof}
The $E_2$ page of the Leray spectral sequence computing $H^* (\Poly_n^{\bd}(X),\Q)$ is given by Theorem~\ref{theorem:coho}.  The assumption that this spectral sequences degenerates on the $E_2$ page thus implies that
for $k\geq 0$, when all the $d_i$ are sufficiently large (given $m,n,r,k$), then $\dim_{\Qb} H^k (\Poly_n^{\bd}(X),\Q)$ is the dimension of the of total degree $k$ part of
$$
           \Sym_{gr}^* H^*(X;\Qb[2r(mn-1)-1])\otimes\bigotimes_{i=1}^m \Sym_{gr}^{*} H^*(X; \Qb[0]).
$$
Recall that for a positively graded vector space $V$ with $V_i$ the degree $i$ part, that the Poincar\'{e} series of $\Sym_{gr}^* V$ is
$$
\prod_{i\geq 1} (1 - (-1)^i t^i )^{-(-1)^i\dim V_i}.
$$
The theorem then follows from the multiplicativity of  Poincar\'{e} series under tensor product.
\end{proof}

Theorem~\ref{theorem:main} then follows  from Theorems~\ref{C:formula} and \ref{T:getvannew}.

\subsection{Application 2: Coincidences for Euler characteristics}

In this subsection we apply Theorem~\ref{theorem:coho} to deduce Claim 1 of Theorem~\ref{theorem:second}. By Hopf's Theorem that the Euler characteristic of a complex is the Euler characteristic of its homology, we have, for any $m$,
    \begin{align*}
     \chi(\Poly^{\bd}_n(X))=\chi(E^{\ast,\ast}_2(X,\bd,n)).
    \end{align*}
When $\dim X=2r$ is even, let $g=2r(mn-1)-1$, and Theorem~\ref{theorem:coho} gives
\begin{align*}
\sum_{\bd\in \Zb_{\geq0}^m} P(E^{\ast,\ast}_2(X,\bd,n)) x_1^{d_1}\cdots x_m^{d_m}
&=\prod_{i\geq 0} (1-(-t)^{g+i}(x_1\cdots x_m)^n )^{(-1)^i b_i(X) }  \prod_{k=1}^m \prod_{i\geq 0} (1 - (-t)^i x_k )^{ (-1)^{i+1}b_i(X)},
\end{align*}
where $b_i(X)=\dim H^i(X,\Q).$  Setting $x_i=x$ and $t=-1$ gives
\begin{align*}
\sum_{\bd\in \Zb_{\geq0}^m} \chi(\Poly^{\bd}_n(X)) x^{|\bd|}
&=\prod_{i\geq 0} (1-x^{mn} )^{(-1)^i b_i(X) }  \prod_{k=1}^m \prod_{i\geq 0} (1 -  x )^{ (-1)^{i+1}b_i(X)}.
\end{align*}
It is standard that
$$
\sum_{d\geq 0} \chi(\Sym^d X) x^d=(1-x)^{-\chi(X)},
$$
and thus Claim 1 of Theorem~\ref{theorem:second} follows.

\subsection{Application 3: Coincidences for Hodge-Deligne polynomials}
\label{section:mixedhodge}

\para{Mixed Hodge structures and Hodge-Deligne polynomials}
We will need some of the basics of mixed Hodge structures; see, e.g., Chapters 3 and 4 of
\cite{PS}, as well as \cite{Sa}.

Let $V$ be a finite-dimensional vector space over $\Q$.  A (rational) {\em pure Hodge structure of weight $n$} on $V$ is a decomposition
\[V_\C:=V\otimes_\Q\C=\bigoplus_{p+q=n}V^{p,q}\]
so that $V^{q,p}=\overline{V^{p,q}}$.  This decomposition gives a decreasing {\em Hodge filtration} $F^iV_\C:=\oplus_{p\geq i}V^{p,q}$ of $V$.  Classical Hodge theory shows that for any
smooth, projective (complex) algebraic variety
$X$, the vector space $H^n(X;\Q)$ has a pure Hodge structure of weight $n$.

Even when $X$ is not compact (and in fact not even assumed to be smooth),  Deligne proved that for each $i\geq 0$, the vector space $H^i(X;\Q)$ comes equipped with a {\em mixed Hodge structure} : there is an ascending {\em weight filtration}
\[0=W_{-1}\subseteq W_0\subseteq \cdots \subseteq W_{2i}=H^i(X;\Q)\]
and a descending {\em Hodge filtration}
\[H^i(X;\C)=F^0\supseteq F^1\supseteq\cdots\supseteq F^m\supseteq F^{m+1}=0\]
with the property that the filtration induced by $F$ on each graded piece $\Gr_n(W):=W_n/W_{n-1}$ is a pure Hodge structure of weight $n$.  Define $h^{p,q,i}(X)$ to be the dimension of the $p^{th}$ graded piece of this $F$-induced filtration on $\Gr_{p+q}(W)$.   The {\em Hodge-Deligne number} $h^{p,q}(X)$ is then defined as $h^{p,q}(X):=\sum_{i\geq 0}h^{p,q,i}(X)$.  Each of these numbers is finite, and only finitely many of them are nonzero.  The {\em Hodge-Deligne polynomial} $\HD_X(u,v)$ of $X$ is the generating function :

\[\HD_X(u,v):=\sum_{p,q\geq 0}h^{p,q}(X)u^pv^q \in \Z[u,v].\]

It is more common to define the Hodge-Deligne polynomial with compactly supported cohomology, but we use regular cohomology here for simplicity. Since we consider smooth varieties,
by Poincar\'{e} duality the data in the polynomial is the same as in the usual definition.

Deligne proved that mixed Hodge structures are {\em functorial} : for any algebraic map $f:X\to Y$ between varieties, the induced map $f^*:H^*(Y;\Q)\to H^*(X;\Q)$  strictly preserves mixed Hodge structures.  Deligne also proved that the Kunneth isomorphism $H^*(X\times Y)\iso H^*(X)\otimes H^*(Y)$ is compatible with mixed Hodge structures, as are cup products.

\begin{corollary}\label{cor:HDpoly}
    Fix positive integers $m,n$ with $mn\geq 2$. Let $X$ be a smooth, connected complex variety.   Then
    \[ \lim_{\bd\rightarrow\infty}\frac{\HD_{\Poly^{\bd}_n(X)}(u,v)}{\HD_{\Sym^{\bd}(X)}(u,v)}
    =\prod_{p,q\ge 0}\prod_{i=0}^{2r}(1-(-1)^{i+2r(mn-1)-1}u^{p+r(mn-1)}v^{q+r(mn-1)})
    ^{-(-1)^{i+2r(mn-1)-1}h^{p,q,i}(X)}
   \]
   In particular, this limit depends only on the product $mn$ and on the mixed Hodge structure on $H^*(X;\Q)$.  Here $\lim_{\bd\rightarrow\infty}$ means ``as all $d_i\rightarrow\infty$'' (at any rates), and we take the limit in the ring of formal power series $\Zb[[u,v]]$ with the usual adic topology.
    \end{corollary}

\begin{proof}[Proof of Corollary~\ref{cor:HDpoly}]
We first claim that
\begin{equation}
\label{eq:mixedhodge0}
\lim_{\bd\to\infty}\HD_{\Poly^{\bd}_n(X)}(u,v)
=\lim_{\bd\to\infty}\HD_{E_2(X,\bd,n)}(u,v)
\end{equation}
in $\Z[[u,v]]$.  To see this, first note that the associated graded (with respect to the two filtrations of the mixed Hodge structure) spectral sequence $E^{p,q}_*(X,\bd,n)$ breaks up as a direct sum of spectral sequences according to weights $(r,s)$.  Corollary 3.2.15 of \cite{De} implies that for any fixed weights $(r,s)$, only finitely many terms of the $E_2$ page of the $(r,s)$-weight part of $E^{p,q}_2(X,\bd,n)$ are nonzero.  Thus the $E_\infty$ page of this part equals the $E_N$ page for some $N$.  We can thus apply Hopf's theorem for finite chain complexes, that the alternating sums of the ranks of the $i$-chains equals the corresponding sum for the ranks of the $i$-dimensional homology groups.  Apply this to each page gives the claim for each weight $(r,s)$.  Applying this one weight at a time, the definition of the adic topology on $\Z[[u,v]]$ gives \eqref{eq:mixedhodge0}.

Equation \eqref{eq:mixedhodge0} implies that :

\begin{equation}
\label{eq:mixedhodge5}
 \lim_{\bd\rightarrow\infty}\frac{\HD_{\Poly^{\bd}_n(X)}(u,v)}{\HD_{\Sym^{\bd}(X)}(u,v)}=
  \lim_{\bd\rightarrow\infty}\frac{\HD_{E_2(X,\bd,n)}(u,v)}{\HD_{\Sym^{\bd}(X)}(u,v)}
  \end{equation}

We thus need to understand $\HD_{E^{p,q}_2(X,\bd,n)}(u,v)$.  We will do this by quoting Statement \ref{cohoodd3}
of Theorem~\ref{theorem:coho}.  We will build up to this, starting with a general statement.

Let $V$ be a graded vector space endowed with a rational mixed Hodge structure.  Then we can write the associated graded of $V$ with respect to the two filtrations as :
\begin{equation}
\label{eq:mixedhodge1}
\gr V=\bigoplus_{p,q\geq 0}\bigoplus_{i\geq 0}(\Q(p,q)[i])^{v_{p,q,i}}
\end{equation}
for some $v_{p,q,i}\geq 0$.  The mixed Hodge structure on $V$ induces a mixed Hodge structure on the symmetric algebra $\Sym^*(V)$.  Its Hodge-Deligne polynomial can be written :
\begin{equation}
\label{eq:mixedhodge2}
\HD_{\Sym^*(V)}(u,v)=\prod_{p,q\geq 0}\prod_{i\geq 0}(1-(-1)^iu^pv^q)^{-(-1)^iv_{p,q,i}}.
\end{equation}

We now apply this general reasoning to
\[V=H^*(X;\Qb[2r(mn-1)-1](r(mn-1),r(mn-1))).\]
To this end, we first need to express $V$ as in \eqref{eq:mixedhodge1}.  Keeping track of weights and degrees, and using the fact that $h^{p,q}(X)=0$ for $p<0$ or $q<0$, we obtain:

\[
\begin{array}{ll}
V&=H^*(X;\Qb[2r(mn-1)-1](r(mn-1),r(mn-1))\\
&\\
& =\bigoplus_{p,q\geq 0}\bigoplus_{i=0}^{2\dim X}h^{p-r(mn-1),q-r(mn-1)}(H^i(X;\Q)[2r(mn-1)-1])\\
&\\
&=\bigoplus_{p,q\geq r(mn-1)}\bigoplus_{i=2r(mn-1)-1}^{2\dim X+2r(mn-1)-1}(\Q(p,q)[i])^{h^{p-r(mn-1),q-r(mn-1), i-2r(mn-1)-1}(X)}
\end{array}
\]

Now Statement \ref{cohoodd3} of Theorem~\ref{theorem:coho} gives $E_2^{p,q}(X,\bd,n)$ as a tensor product of two symmetric algebras.  As $\bd\to\infty$, the Hodge-Deligne polynomial of the first of these algebras converges to $\HD_{\Sym^*(V)}(u,v)$, while for the second of these algebras
the Hodge-Deligne polynomial converges to $(\HD_{\Sym^\infty(X)}(u,v))^m$.  Since
\[\HD_{U\otimes W}(u,v)=\HD_U(u,v)\HD_W(u,v)\]
for mixed Hodge structures $U$ and $W$, the right-hand side of \eqref{eq:mixedhodge5}
equals $\HD_{\Sym^*(V)}(u,v)$.   Plugging the formula we just obtained for $V$ in \eqref{eq:mixedhodge2} gives:
  \[\HD_{\Sym^*(V)}(u,v)= \prod_{p,q\geq r(mn-1)}\prod_{i=2r(mn-1)-1}^{i=2\dim X+2r(mn-1)-1}(1-(-1)^iu^pv^q)^{-(-1)^ih^{p-r(mn-1),q-r(mn-1), i-2r(mn-1)-1}(X)}\]
  which by re-indexing equals
  \[=\prod_{p,q\geq 0}\prod_{i=0}^{2\dim X}(1-(-1)^{i+2r(mn-1)-1}u^{p+r(mn-1)}v^{q+r(mn-1)})^{-(-1)^{i+2r(mn-1)-1}h^{p,q,i}(X)}
\]
thus giving the theorem.

\end{proof}

\section{The poset of colored $n$-equals partitions and its homology}
\label{section:nequals}

The goal of this section is to prove Proposition~\ref{proposition:kunneth} and Theorem \ref{thm:EL} below.  These purely combinatorial results are the first of three main ingredients in our proof of Theorem \ref{theorem:coho}.

Fix throughout this section an integer $m\geq 1$ and $m$ colors.  Let $D$ be a {\em finite colored set}; that is, a finite set $D$ and function $D\to \{1,\ldots ,m\}$ which we think of as ``coloring'' each point of $D$ by one of the colors $1,2,\ldots ,m$.  Let $D(i)$ denote the subset of $D$ consisting of all elements with color $i$. For any set $S$ denote the cardinality of $S$ by $|S|$.    Let
\[
  S_{D}:=\{\textrm{color preserving self-bijections of $D$} \}
  \]

\begin{definition}[{\bf \boldmath$n$-equals partition}]
\mbox{}
    \begin{enumerate}
    \item Fix $m$ colors.  Let $D$ be a finite colored set.  A partition of $D$ is an {\em $n$-equals} partition if each block of the partition either has size $1$, or contains at least $n$ elements of each of the $m$ colors.
            \item Denote by $\Pi^D_n$ the poset of $n$-equals partitions of the colored set $D$, ordered by refinement: $I\leq J$ if and only if $I$ refines $J$.
    \end{enumerate}

    \end{definition}

For $m=1$, the lattice $\Pi^d_n$ has been intensely studied by Bj\"orner and his collaborators, under the name of ``$n$-equals'' arrangement. Recall that elements of $\Pi^d_n$ are partitions $I$ of $\{1,\ldots,d\}$ such that all blocks in the partition have size $1$ or size at least $n$.

We record the following elementary observation.   For a poset $P$, denote by
$P(\leq x)$ the subposet of $P$ consisting of all elements of $P$ that are $\leq x$.

\begin{lemma}
\label{L:fcprod1}
Let $J\in \Pi^D_n$.  Let $J_1,\ldots ,J_k$ denote the blocks of the partition $J$.
Let $\stab_J\subset S_D$ be the stabilizer of $J$.
There exists a $\stab_J$-equivariant isomorphism of posets
\[ \Pi_n^D(\leq J)\cong \prod_i\Pi_n^{J_i}.\]
\end{lemma}
\begin{proof}
Refinements of $J$ are equivalent to a choice of $n$-equal partitions of $J_i$ for each $i$.
The claim follows.
\end{proof}

We will need the following.
\begin{definition}[{\bf \boldmath$\hat{0}$ and \boldmath$\hat{1}$}]
    Given a poset $P$ with an initial object $\hat{0}$ and a terminal object $\hat{1}$, define $\overline{P}:=P\setminus\{\hat{0},\hat{1}\}$.
\end{definition}

\begin{definition}[{\bf The order complex \boldmath$\Delta(P)$ of a poset \boldmath$P$}]
For a poset $P$, the {\em order complex} $\Delta(P)$ associated to $P$ is the simplicial complex whose $k$-simplices are the chains $x_0<x_1<\cdots <x_k$  (i.e. the totally ordered subsets of $P$).
\end{definition}

\begin{convention}
\label{convention:empty}
Note that if $P$ is the poset with two elements ($\hat{0}, \hat{1})$ then $\overline{P}$
is empty, so $\widetilde{H}_*(\Delta(\overline{P});\Z) = \Z$ in degree $-1$. As a special
convention, if $P$ is the poset with one element then we will say that $\widetilde{H}_*(\Delta(\overline{P});\Z) = \Z$ in degree $-2$.
\end{convention}

\subsection{EL-shellability of the colored $n$-equals lattice}

We quickly recall the theory of lexicographic shellability, first developed by Bj\"{o}rner-Wachs; see, e.g., Section 5 of \cite{BW}.

For a poset $P$ let $E(P)$ denote the set of {\em edges} of $P$, i.e. pairs $a,b\in P$ with $a<b$ and no $c$ such that $a<c<b$.  For $a\leq b$ in $P$, the {\em (closed) interval} is defined as
$[a,b]:=\{x\in P : a\leq x\leq b\}$; the open interval $(a,b)$ is defined similarly.  The poset $P$ is {\em bounded} if it has a greatest element $\hat{1}$ and a least element $\hat{0}$.  A {\em chain} of length $r$ in $P$ is a string $a_0<\cdots <a_r$ with $a_i\in P$.  A chain is {\em maximal} if
it is not a proper subchain of any chain in $P$.  A chain $a_1<\cdots <a_r$ is {\em unrefinable} if it is maximal in the interval $[a_1,a_r]$.

\
An {\em edge-labelling} of $P$ is a map $\lambda:E(P)\to \Lambda$ for some poset $\Lambda$. Given an edge-labelling $\lambda$, for a chain $c$ of length $r$  we write $\lambda(c)\in \Lambda^r$ for the ordered tuple of the labels of the edges of $c$. A {\em rising chain} in an
 interval $[a,b]$ is a chain $c$ with the property that it is a maximal chain  in $[a,b]$  with  $\lambda(c)=(\ell_1,\ldots ,\ell_r)$ satisfying
$\ell_1<\cdots <\ell_r$ in $\Lambda$.

\begin{definition}[{\bf EL-labelling and EL-shellability}]
An edge-labelling $\lambda$ of a poset is called an {\em EL-labelling}
if :
\begin{enumerate}
\item Every 
interval 
$[a,b]$
has a unique rising chain $c$, and
\item this unique rising chain is lexicographically strictly first among maximal chains: $\lambda(c)<\lambda(c')$ for all other maximal chains $c'$ in $[a,b]$.
\end{enumerate}

A bounded poset that admits an EL-labeling is called {\em EL-shellable}.
\end{definition}

The property of EL-shellability is preserved by several standard properties of posets.  In particular, as given in Theorem 10.16 of \cite{BW2}: if $P$ and $Q$ are bounded posets then $P$ and $Q$ are EL-shellable if and only if $P\times Q$ is EL-shellable.
The importance of EL-shellability comes from a theorem of Bj\"{o}rner-Wachs stating that if a bounded poset $P$ is EL-shellable then $\Delta(\overline{P})$ has the homotopy type of a wedge of spheres, indexed as follows.

\begin{definition}[{\bf Falling chains}]
    Let $P$ be a poset with an edge-labeling.  A chain $a_0<\cdots <a_r$ of $P$ is \emph{falling} if it is maximal
     and for all $0< i \leq r$  the label $\ell_i$ of $a_{i-1} <a_i$ is not less than the label $\ell_{i+1}$ of $a_{i} <a_{i+1}$.
\end{definition}

\begin{lemma}[Theorem 5.9 of \cite{BW}]
\label{L:BW}
If a poset $P$ is $EL$-shellable then
   \begin{equation*}
        \widetilde{H}_r(\Delta(\overline{P});\Zb)\cong \Zb^{\{\text{falling chains of $P$ of length }r+2\}} .
   \end{equation*}
  \end{lemma}

The main result of this subsection is the following.

\begin{theorem}[{\bf  \boldmath$\Pi^{D}_n$ is EL-shellable}]
\label{thm:EL}
    The poset $\Pi^{D}_n$ is EL-shellable.
\end{theorem}

\begin{remark}
Bj\"orner--Wachs \cite[Theorem 6.1]{BW} give an EL-labelling for $\Pi^d_n$ for a single $d$ (i.e. in the case $m=1$).  By refining their construction, we give an EL-labelling of $\Pi^{D}_n$ in all cases $m\ge 1$.  Our labelling (and proof) reduces to theirs in the case $m=1$.  The case $m>1$ is considerably more delicate.
\end{remark}

\begin{proof}[Proof of Theorem~\ref{thm:EL}]
If $m=n=1$, note $\Pi^{D}_n$ is the same as if $m=1$ and $n=2$, so without loss of generality, we assume $mn\geq 2$.
First, we determine all the edges in $\Pi^{D}_n$.
     Extending Bj\"orner--Wachs, we  introduce three types of edges in $\Pi^{D}_n$, which we will show are exhaustive.
Whenever we write $a_i\in D$ it denotes that $a_i$ is an element of $D$ of color $i$.

\bigskip
\noindent
{\bf Block creation: }   A new non-singleton block $B$
with $n$ elements each of the $m$ colors is created from singletons.  Let $B_i$ be the set of elements of color $i$ in the block.

\medskip
\noindent
{\bf Singleton adding: }A singleton block $\{a_i\}$ is merged with a non-singleton block.

\medskip
\noindent
{\bf Block merging: } Two non-singleton blocks $B$ (whose subset of color $i$ elements is $B_i$) and $C$  (whose subset of color $i$ elements is $C_i$)  are merged.

\bigskip
    By induction, we can show that the first two types of edges are sufficient to generate every element of $\Pi^{D}_n$. However,
 the first two types of edges do not give all edges.

  To see that the 3 types above exhaust the list of edges, suppose that $I<J$ is an edge.  Consider a non-singleton block of $J$ that is not a block of $I$.  The block of $J$ is either (1) entirely singletons in $I$, (2) contains a non-singleton block of $I$ and a singleton of $I$, or (3) contains at least two non-singleton blocks of $I$.  In each case, there exists $I'$ with $I< I' \leq J$ and $I<I'$ of the corresponding type, and so $I'=J$ since $I<J$ is an edge.

    For each $1\leq i\leq m$, we pick a linear ordering on $D(i)$ so that it is an ordered set. Let $d_i:= \max D(i)$.
We let $\overline{D(1)}$ be an isomorphic copy of the ordered set $D(1)$ with elements $\bar{a}$ for $a\in D(1)$.
Let $D(1)^\epsilon$ be the ordered set whose elements at $a\in D(1)$ and $a-\epsilon$ for $a\in D(1)$ with the obvious ordering
($a-\epsilon<a$ and if $a>b$ for $a,b\in D(1)$, then $a-\epsilon>b$ ).
   Let $\vec{\Lambda}$ be the poset
    \begin{equation*}
        \vec{\Lambda}:= \overline{D(1)} \sqcup D(1)^\epsilon \times D(2) \times D(3) \times\ldots\times D(m),
    \end{equation*}
    where $D(1)^\epsilon\times\ldots\times D(m)$ is ordered lexicographically, and where $\bar{a}<\gamma$ for all $\bar{a} \in \overline{D(1)}$ and all $\gamma\in D(1)^\epsilon\times\ldots\times D(m)$.

    We now define an edge-labeling of $\Pi^{D}_n$ with labels in $\vec{\Lambda}$:
    \begin{enumerate}
        \item For block creation of a block $B$, we assign the label $(\max B_1, \ldots,\max B_m)$.
        \item For adding a singleton $a_i$ of color $i>1$ to a block $B$, we assign the label
        \begin{align*}
                \lambda(B\cup\{a_i\})&:=(\max B_1,\ldots,\max B_{i-1},a_i,d_{i+1},\ldots,d_m).\intertext{For adding a singleton $a_1$, we assign the label}
                \lambda(B\cup\{a_1\})&:=(a_1-\delta \epsilon,d_2,\ldots,d_m)
            \end{align*}
            where $\delta=0$ if $a_1>\max B_1$ and $\delta=1$ if $a<\max B_1$.
        \item For block merging, we assign the label $\overline{\max\{B_1\cup C_1\}}$.
    \end{enumerate}

    We now prove that this edge-labeling is an $EL$-labeling.
Note that for an edge-labeling to be an EL-labeling, it suffices to require that every interval $[x,w]$ has a unique rising chain, and that the first edge $x<y$ of that rising chain has label less than any other   edge $x<z$ with $z \leq w$.
We will prove this stronger condition of our edge-labeling.
        We first prove this  condition for terminal intervals $[x,\hat{1}]$ in $\Pi^{D}_n$, and then prove it for general intervals.

\bigskip
\noindent
    {\bf Step 1: Terminal Intervals.}
    \noindent
        Let $x\in\Pi^{D}_n$ be any element.  We will show the condition above for $[x,\hat{1}]$.
        We will consider three cases based on the number of non-singleton blocks of $x$.

{\bf Case I:} We first consider the case that $x$ has a single non-singleton block $B$.   By our edge-labeling, in a rising chain, block merging can never occur after block creation.  Yet, any maximal chain of $[x,\hat{1}]$ that creates a block must later merge it with the block of $x$.  Thus, any rising chain  in $[x,\hat{1}]$ can only consist of adding singletons.

        \begin{claim}\label{addsing}
            There is a unique rising order in which singletons can be merged to $B$.
        \end{claim}
        \begin{proof}[Proof of Claim \ref{addsing}]

            Let $S_{< B}\subset D\setminus B$ consist of all singletons $a_i$  (for $1\le i\le m$) such that $a_i<\max B_i$.
            Let $S_{> B}\subset D\setminus B$ consist of all singletons $a_i$   (for $1\le i\le m$) such that $a_i>\max B_i$.
            Note that $S_{< B}\cup S_{>B}=D\setminus B$. We first will prove that in a rising chain in $[x,\hat{1}]$
            singletons in $S_{< B}$ must be added before singletons in $S_{>B}$.

            For any subset $C$ of $D$, we can partition $D\setminus C$ into $S_{< C}\cup S_{>C}$ as above.
           If we have a block $C$  and add a singleton $a_i\in S_{<C}$, we call that a \emph{low singleton add},
                and if we have a block $C$  and add a singleton $a_i\in S_{>C}$, we call that a \emph{high singleton add}.
            Note that the label
            $\lambda(C \cup\{a_i\})$ of a high singleton add to form a block $C'=C \cup\{a_i\}$
            is greater than or equal to the label of creating the resulting block $C'$.  However, the label $\lambda(C' \cup\{b_i\})$
            of a low singleton add of $b_i\in S_{<C'}$ to a block $C'$ to form a block $C''=C' \cup\{b_i\}$
            is less than the label of creating the starting block $C'$.   Thus, in a rising chain, a high singleton add can never be followed by a low singleton add.   Starting at $x$, the first singleton add from  $S_{>B}$ will be a high singleton add (even if it occurs after some low singleton adds), and any singleton add from $S_{< B}$ at any point will be a low singleton add.
            Thus, in a rising chain in $[x,\hat{1}]$
            singletons in $S_{< B}$ must be added before singletons in $S_{>B}$.

Next we will show there is at most one rising order to merge singletons in $S_{< B}$,  and at most one rising order to add singletons in $S_{>B}$ to $B$.

            {\bf Singletons in $S_{< B}$:} We claim there is a unique rising order in which singletons in $S_{< B}$ can be added to $B$.
If we have added some singletons from $S_{< B}$ to $B$ to form a block $B'$, then note that $\max B_i=\max B'_i$ for all $i$.
Thus, whenever we add a singleton $a_i$ to a block formed from adding singletons from $S_{<B}$ to $B$, the label will be
\begin{equation}\label{E:lowsa}
(\max B_1,\ldots,\max B_{i-1},a_i-\delta \epsilon,d_{i+1},\ldots,d_m)
\end{equation}
(where $\delta=0$ if $i>1$ and $\delta=1$ if $i=1$). If $i<i'$ and $a_i,b_{i'}\in S_{<B}$, then since $a_i<\max B_i$,
$$
(\max B_1,\ldots,\max B_{i-1},a_i-\delta \epsilon,d_{i+1},\ldots,d_m) <
(\max B_1,\ldots,\max B_{i'-1},b_{i'}-\delta \epsilon,d_{i'+1},\ldots,d_m).
$$ So if $i<i'$, in a rising chain all singletons of color $i$ in $S_{<B}$ must be added before any singletons of color $i'$ in $S_{<B}$.  Among singletons of a given color in $S_{<B}$, we note by the labels as given in Equation~\eqref{E:lowsa}, in a rising chain they must clearly be added in increasing order.

  {\bf Singletons in $S_{> B}$:}
After adding all the singletons in $S_{<B}$ to $B$ to form a block $B'$, we will show that there is a unique rising order in which to add the singletons of $S_{>B}=S_{>B'}.$  The first singleton add from $S_{>B}$ will be a high singleton add, and so by our observation above, all further singleton adds in a rising chain must be high singleton adds.  When we add a singleton $a_i$ from $S_{> B}$ to block $C$ to form block $C'$ the label will be
\begin{equation}\label{E:highsa}
(\max C_1,\ldots,\max C_{i-1},a_i,d_{i+1},\ldots,d_m).
\end{equation}
(Note we do not have to consider subtracting $\epsilon$ since that is only required for low singleton adds.)
We have $\max C_j= \max C'_j$ for $j\ne i$ and $\max C'_i=a_i$ because adding $a_i$ was a high singleton add.  Suppose for the sake of contradiction that in a rising chain we add $a_i$ of color $i$ followed by $b_{i'}$ of color $i'$ with $i<i'$.  Then the label of adding $b_{i'}$ to $C'$ is
\begin{equation*}
(\max C_1,\ldots,\max C_{i-1},a_i,\max C_{i+1}, \dots, b_{i'},  d_{i'+1},\ldots,d_m).
\end{equation*}
Since every entry past the $i$th in this label is at most the corresponding entry in the label of Equation~\ref{E:highsa} above,
the chain cannot be rising and we have a contradiction.  Thus, for $i<i'$ in a rising chain we have to add all elements from
$S_{> B}$ of color $i'$ before each of color $i$.  Within a color, in a rising chain singletons clearly have to be added in increasing order.

   So far, we have shown that in any rising chain in $[x,\hat{1}]$, we have to add in singletons from $S_{<B}$ in a unique order, and then add in singletons from $S_{>B}$ in a unique order.  It remains only to check that adding the singletons in this order
does indeed give a rising chain.     It is easy to check that adding singletons from $S_{<B}$  in the required order is rising, and the final label is less than the label of the block creation of $B$.  The first addition of a singleton from  $S_{>B}$ to $B\cup S_{<B}$ has label greater than the label of the block creation of $B$.  Finally, it is easy to check that adding singletons from $S_{>B}$ to a block $B\cup S_{<B}$ in the required order is indeed rising.
    \end{proof}

  We we have shown that when $x$ has a single non-singleton block $B$, there is a unique rising chain in $[x,\hat{1}]$.
  Next we will show that the first edge $x<y$ of this chain has label less than any other edge $x<z$.
Suppose, for the sake of contradiction, that there is some edge $x<z$ with $z\ne y$ and $\lambda(x<z)\leq \lambda(x<y)$.

1. First we consider the case that $x<z$ is a block creation of a block $C$ (disjoint from $B$).  Then $\max C_1 \in D\setminus B$.

 a.  We consider the case that $x<y$ has label with first coordinate $\leq \max B_1$, then we have $\max C_1 \leq \max B_1$, which implies $\max C_1 < \max B_1$.   Thus, from our description of the unique rising chain in $[x,\hat{1} ]$ above, since $S_{<B}$ has some element of color $1$ in it, the edge $x<y$ adds the smallest singleton of  color $1$ to $B$.
Thus since $\max C_1$ is some singleton of color $1$ in $S_{<B}$, the first coordinate of $\lambda(x<y)$ is at most $\max C_1 -\epsilon$, which contradicts the hypothesis that $\lambda(x<z)\leq \lambda(x<y)$.

 b.  We consider the case that $\lambda(x<y)$ has first coordinate $>\max B_1$, i.e. $x<y$ is a high singleton add of a singleton of color $1$, and by our above analysis of the unique rising chain in $[x,\hat{1} ]$, it must be adding the smallest singleton $s_1$ of color $1$ in $D\setminus B$ and $D\setminus B$ must have no elements of any color $\geq 2$.  If $m\geq 2$, then there can be no block $C$ in
 $D\setminus B$ to create.  If $m=1$, then since $n\geq 2$, we have $\max C_1 >s_1$.  Then $\lambda(x<y)$ has first coordinate $s_1< \max C_1$, which is a contradiction since $C_1$ is the first coordinate of   $\lambda(x<z)$.

 2.  Second, we consider the case that $x<z$ adds a singleton $b_j$ to $B$.

 a.  We consider the case that $b_j\in S_{<B}$. Then $S_{<B}$ is non-empty, and we see from our analysis above of adding singletons in $S_{<B}$ that the label of $x<z$ is the same as the label of the edge that adds $b_j$ in the rising chain, which is greater than $\lambda(x<y)$, which is a contradiction.

 b.   We consider the case that $b_j\in S_{>B}$, i.e. $x<z$ is a high singleton add. Then $\lambda(x<z)$ is greater than the label of the block creation of $B$.  If $x<y$ was a low singleton add, then $\lambda(x<y)$ is less than the label of the block creation of $B$, and so we conclude $x<y$ must be a high singleton add.  From our analysis above of the unique rising chain, we then have $S_{<B}$ is empty and $x<y$  adds the minimal element $a_i$ of the maximal color $i$ of $S_{>B}$.  Thus $j\leq i$, and if $j=i$ then $a_i< b_j$.    Since $x<z$ and $x<y$ are both high singleton adds, we see from the definition of the labels that $\lambda(x<y)<\lambda (x<z)$.

Thus in every case, we conclude that the first edge $x<y$ of the rising chain in $[x,\hat{1}]$ has label less than any other edge $x<z$.

    {\bf Case II:}  Suppose $x$ had no non-singleton blocks, i.e. $x=\hat{0}$.  Then the only edges $x<y$ are block creations.
    Consider a rising chain $c$ in $[x,\hat{1}]$ that starts with an edge $x<y$ that creates a block $B$. From the above, we see there is a unique rising chain in $[y,\hat{1}]$ that we can append to $x<y$ to obtain $c$.  If $S_{<B}$ is not empty, then the  rising chain in $[y,\hat{1}]$ starts with a low singleton add whose label is less then the label of the block creation of $B$, which is a contradiction.  Thus $S_{<B}$ is empty, and we have that $B$ must consist of the $n$ smallest elements of each color.  Thus any rising  chain in $[x,\hat{1}]$ must start with a specified block creation $x<y$, and from there continue with the unique rising  chain in $[y,\hat{1}]$.  This shows there is at most $1$ rising  chain in $[x,\hat{1}]$.  Also, note that if we form the block $B$ of the $n$ smallest elements of each color in $x<y$, then $S_{<B}$ is empty, and the first edge in the rising chain of $[y,\hat{1}]$ is a high singleton add, which implies its label is greater than the label of the block creation of $B$.  So, we can concatenate $x<y$ with the rising chain of $[y,\hat{1}]$ to get a rising chain of $[y,\hat{1}]$.  Since any edge $x<y$ is a block creation, and the block of the $n$ least elements of each color has label less than any other block creation, the first edge of the rising chain in $[x,\hat{1}]$ has label less than any other edge $x<z$.

      {\bf Case III:}  Suppose $x$ has more than one non-singleton block.  Then in any maximal chain of $[x,\hat{1}]$, there must be some block merging, and in a rising chain, all block merging must happen before any block creation or singleton adding.  So any rising chain must start with block merging until there is only one block.  If there are $k$ blocks, and the maximal elements of color $1$ in them are $a(1)_1<\dots< a(k)_1,$ then note that $\{\overline{a(j)}_1\}$ are the only possible labels of block merges starting from these blocks into a single block.   Further, $\overline{a(1)}_1$ can never be one of the labels of the block merges.  Thus, since $k-1$ block merges are required with strictly increasing labels, a rising chain must first merge the blocks containing $a(1)_1$ and $a(2)_1$, and then merge the result with the block containing $a(3)_1$, and so on.  Once we have one block in a partition $w$, there is a unique rising chain in $[w,\hat{1}]$ from the above that only involves singleton adds.  Therefore there is at most 1 rising chain in $[x,\hat{1}]$.  Further, we note that if we merge blocks in the order described above, that part of the chain is rising, and the unique rising chain in $[w,\hat{1}]$ only involves singleton adds, whose labels are all greater than block merge labels, and thus there is a rising chain in $[x,\hat{1}]$.
Finally, we consider the first block merge $x<y$ in this rising chain and any other edge $x<z$.  If $z$ is not a block merge, then clearly $\lambda(x<z) > \lambda(x<y)$, and if $z$ is a block merge other than that of the blocks containing $a(1)_1$ and $a(2)_1$, then $\lambda(x<z) > \lambda(x<y)$.  So, the first edge of the rising chain in $[x,\hat{1}]$ has label less than any other edge $x<z$.

\bigskip
\noindent
    {\bf Step 2: General Intervals.}
    \noindent
     Let $x<w\in\Pi^{D}_n$ be a general pair of elements.
     Let the non-singleton blocks of $w$ be $D_\alpha$, indexed by $\alpha$.  Then, Lemma \ref{L:fcprod1} gives that $\Pi^{D}_n(\leq w) \simeq \prod_\alpha
     \Pi^{D_\alpha}_n$.
Each edge  in   $\Pi^{D}_n(\leq w)$ only modifies blocks that are subsets of a  single $D_\alpha$.
     For each $D_\alpha$, let $x_\alpha$ be the partition of $D_\alpha$ obtaining by taking the blocks of $x$ that are subsets of $D_\alpha$.  Then, for each $\alpha$, we know there is a unique rising chain
in $[x_\alpha,\hat{1}]\subset \Pi^{D_\alpha}_n$.
 Note that  edge labels only depend on the blocks that are being modified by an edge.
 If we have a rising chain of $[x,w]$, for any $\alpha$, we can take the subset of edges that involve modifying subsets of $D_\alpha$ and obtain a rising chain of $[x_\alpha,\hat{1}]\subset \Pi^{D_\alpha}_n$ (which must be the unique such chain).
 Further,  the first coordinate of the label of any edge gives an element in a block (singleton or non-singleton) being modified by that edge.  Thus if $\alpha\neq \beta$, a label of an edge modifying subsets of $D_\alpha$ cannot be the same as the label of an edge modifying subsets of $D_\beta$.  Since the labels are linearly ordered, there is a unique way to combine the rising chains of $[x_\alpha,\hat{1}]\subset \Pi^{D_\alpha}_n$ into a rising chain of $[x,w]$.

 Finally, suppose that $x<z$ is some edge with $z\leq w$.
Let $\beta$ be such that the first edge  of  the rising chain of $[x,w]$ modifies subsets of $D_\beta$.
For each $\alpha$, let $x_\alpha< y_\alpha$ be the first edge of the rising chain in  $[x_\alpha,\hat{1}]\subset \Pi^{D_\alpha}_n$.  Note that for each $\alpha$, the label of $x_\alpha< y_\alpha$ occurs as a label in the rising chain of $[x,w]$, and thus $\lambda(x_\alpha< y_\alpha)\geq \lambda(x_\beta< y_\beta)$ with equality if and only if $\alpha=\beta$.
 Then $x<z$ corresponds, for some $\alpha$, to an edge $x_\alpha< z_\alpha$ of elements of $\Pi^{D_\alpha}_n$, and thus either $z_\alpha=y_\alpha$ or the label of $x<z$ (which is the same as the label of $x_\alpha< z_\alpha$) is greater than or equal to the label of $x_\alpha<y_\alpha$, with equality if and only if $z_\alpha=y_\alpha$.
   Thus $\lambda(x<z)\geq \lambda(x_\beta< y_\beta)$ with equality if and only if $\alpha=\beta$ and $z_\alpha=y_\beta$.
   In other words, for any edge $x<z$ with $z\leq w$ that is not the first edge  of  the rising chain of $[x,w]$, we have that
   $\lambda(x<z)$ is greater than the label of the first edge of the rising chain of $[x,w]$.
\end{proof}

Theorem \ref{thm:EL} together with Lemma \ref{L:BW} shows that for any colored $n$-equals partition $I$, the homology of $\Delta(\overline{\Pi^D_n(\le I)})$ is torsion free.  Along with Lemma \ref{L:fcprod1}, this suggests a K\"unneth type formula. We develop this now.

For any colored subset $F$ of $D$, the inclusion $F\into D$ induces an injection of lattices $\Pi^{F}_n\into \Pi^D_n$, and under this injection $\Pi^{F}_n\cong \Pi^D_n(\leq \tilde{F})$, where $\tilde{F}$ is the $n$-equals partition consisting of the single nonsingleton block $F$ (or $\tilde{F}=\hat{0}$ if $F$ contains fewer than $n$ elements of some color).

In light of this, we can rewrite the isomorphism of Lemma~\ref{L:fcprod1} as
\[\prod_i\Pi^D_n(\leq \tilde{J_i})\to^\cong\Pi^D_n(\leq J)\]
where $\{J_i\}$ are the blocks of the partition $J$.  Note that we are viewing $J_i$ both as a block of $J$ and as a partition of $D$ with only that non-singleton block.

\begin{definition}[{\boldmath$\cd(I)$}]
\label{definition:cd}
For a partition $I$, define $\cd(I):=|D|-|I|$, i.e. $\cd(I)$ equals the number of elements of $D$ minus the number of blocks of $I$.
\end{definition}

\begin{definition}
    Given $I,J\in\Pi^D_n$, denote their {\em join} by $I\wedge J$, i.e. $I\wedge J$ is the finest partition which both $J$ and $I$ refine. We say that $I$ and $J$ {\em meet transversely} when $\cd(I)+\cd(J)=\cd(I\wedge J)$.
\end{definition}

As discussed in \S 1.8-1.9 and \S 4.2 of \cite{DGM}, the direct sum
\begin{align}\label{posetalg}
    \left(\bigoplus_{I\in\Pi^D_n} \tilde{H}_{\ast-2}(\Delta(\overline{\Pi^D_n(\le I)});\Z)\right)
\end{align}
has the structure of a graded-commutative algebra; we refer to its product as the {\em intersection product}. From the construction (below), this algebra carries a natural $S_D$ action.

Explicitly, the intersection product is given by 0 on summands associated to $I$ and $J$ which do not meet transversely.  On summands associated to $I$ and $J$ which do meet transversely, the intersection product is given on each summand by the composition
\begin{align}\label{intprod}
    \tilde{H}_i(\Delta(\overline{\Pi^D_n(\le I)});\Z)\otimes \tilde{H}_j(\Delta(\overline{\Pi^D_n(\le J)});\Z)&\to^\cong  \tilde{H}_{i+j}(\Delta(\overline{\Pi^D_n(\le I)})\times \Delta(\overline{\Pi^D_n(\le J)});\Z)\nonumber\\
    &\to^\cong \tilde{H}_{i+j+2}(\Delta(\overline{\Pi^D_n(\le I)\times\Pi^D_n(\le J)});\Z)\nonumber\\
    &\to^\wedge\tilde{H}_{i+j+2}(\Delta(\overline{\Pi^D_n(\le I\wedge J)});\Z)
\end{align}
where the first isomorphism is given by K\"unneth, the second isomorphism arises from a canonical homeomorphism of order complexes of bounded posets, due to Walker \cite[Theorem 5.1]{Wa} \footnote{We remark that our notation differs slightly from \cite{Wa}: here a degree $2$ shift appears rather than degree $1$ (as in \cite{Wa}); this is because Theorem 5.1 in \cite{Wa} involves a suspension.}, and the final map comes from the join, viewed as a map of posets
\begin{align*}
    \Pi^D_n(\le I)\times\Pi^D_n(\le J)&\to \Pi^D_n(\le I\wedge J)\\
    (L,K)&\mapsto L\wedge K.
\end{align*}
For a partition $I$ with blocks $I_1,\ldots,I_\ell$, note that the isomorphism of Lemma \ref{L:fcprod1} is just the iterated join
\begin{align*}
    \prod_i \Pi^D_n(\le \tilde{I_i})&\to^\wedge_\cong \Pi^D_n(\le I)\\
    (L_1,\ldots,L_\ell)&\mapsto L_1\wedge\ldots\wedge L_\ell.
\end{align*}
Comparing with the definition of the intersection product, we conclude the following.
\begin{proposition}({\bf A ``K\"unneth'' decomposition})
\label{proposition:kunneth}
    Let $I$ be a partition.  Denote by $I_1,\ldots ,I_\ell$  the blocks of $I$.  The intersection product induces for each $k\geq 0$ a $\stab_I$-equivariant isomorphism
    \begin{equation}\label{kunneth1}
        \bigoplus _{k_1+\cdots +k_\ell=k}\bigotimes_{i=1}^\ell \widetilde{H}_{k_i}(\Delta(\overline{\Pi^{I_i}_n});\Z)\to^{\cong}\widetilde{H}_{k+2(\ell-1)}(\Delta(\overline{\Pi^D_n(\le I)});\Z).
    \end{equation}
    In particular, the algebra \eqref{posetalg} is generated by the subspace $\bigoplus_{I'}H_{*-2}(\Delta(\overline{\Pi^D_n(\leq I')});\Z)$ where the sum is over partitions $I'\in \Pi^D_n$ with precisely one nonsingleton block.
\end{proposition}

One might naively expect that Proposition \ref{proposition:kunneth} should follow directly from
the K\"unneth theorem.   However, it is not the case that the isomorphism in Lemma~\ref{L:fcprod1} holds with $\Pi^D_n(\le I)$ replaced by $\overline{\Pi^D_n(\le I)}$.    Indeed, this is one reason that the intersection product is needed. On the other hand, an interpretation in terms of Kunneth can be found in Lemmas 4.2 and 4.5 of \cite{Pet2}.

\section{Spaces of ordered $0$-cycles}
\label{section:0cycles}
Now fix a manifold $X$ of dimension $N$.  Define $\widetilde{\Poly}^D_n(X)\subseteq X^{D}$
 to be the space of $D$-tuples of (not necessarily distinct) points in $X$ labeled by the elements of $D$ such that no point of $X$ has at least $n$ labels of each color.    Since we have fixed $m$ colors throughout, if $D$ happens to not include any elements of some color then $\widetilde{\Poly}^D_n(X)=X^D$.
The permutation action of $S_D$ on $X^{D}$ leaves invariant $\widetilde{\Poly}^D_n(X)$.

The goal of this section is to prove Theorem \ref{theorem:leray} below.  This theorem
describes the $\Q[S_D]$-algebra structure of the $E_2$-page of the Leray spectral sequence associated to the inclusion $\widetilde{\Poly}^D_n(X)\to X^D$ and the constant sheaf $\Z$.
The description will be in terms of the cohomology of $X$ and the homology of order complexes related to $\Pi^D_n$.
Petersen \cite{Pe} gives a spectral sequence for stratified spaces that might also be used to be understand the homology of $\widetilde{\Poly}^D_n(X)$.  However, we crucially
need to further understand that action of $S_D$ on the homology, and it is not straightforward to see that action in Petersen's spectral sequence.
We begin by studying the combinatorics of the complement of $\widetilde{\Poly}^D_n(\Rb^N)$ in $\Rb^N$. Let
\begin{align*}
    L_1(D):=\{(\vec{x}_1,\ldots,\vec{x}_m)\in(\Rb^{N})^{D(1)}\times\ldots\times(\Rb^{N})^{D(m)}~|~ x_{1i}&=x_{ai}\text{ for } 1\le i\le n,~1<a\le m,\\
    x_{11}&=x_{1j}\text{ for }2\le j\le n\}
\end{align*}

Define the {\em colored $n$-equals arrangement} $\mathcal{A}_{n}^{N,D}$ to be the linear subspace arrangement in $(\Rb^{N})^D$ consisting of the set of all translates of $L_1(D)$ under the action of $S_D$. Denote by $\Pi^D_n(\Rb^{N})$ the associated {\em intersection lattice} :
\begin{equation*}
    \Pi^{D}_n(\Rb^{N}):=\{L\subset (\Rb^{N})^{D}~|~L=L_{\sigma_1}\cap\cdots\cap L_{\sigma_k},\text{ for }\sigma_i\in S_{D},~L_{\sigma_i}=\sigma_i(L_1(D))\}
\end{equation*}
and $\Pi^{D}_n(\Rb^{N})$ is ordered by reverse inclusion. Note also that we include the entire space $(\Rb^{N})^{D}$ (i.e. the empty intersection when $k=0$ above), and will alternately denote it by $\hat{0}$. It is an initial  element of the poset $\Pi^{D}_n(\Rb^{N})$.

\begin{remark}
    For $(m,n)=(1,2)$, the arrangement $\mathcal{A}_{1}^{2,D}$  with complement $\widetilde{\Poly}^D_2(\Cb)$ is precisely the braid arrangement studied by Arnol'd \cite{Ar}. Arnol'd showed that the cohomology algebra is generated by classes in degree 1 subject to a quadratic relation. The algebras $H^\ast(\widetilde{\Poly}^D_n(\Rb^{N});\Zb)$ are near cousins of Arnol'd's algebra, and one might hope they admit a similar presentation, though we do not expect that they are always quadratic algebras.
\end{remark}

\begin{problem}\label{problem:presentation}
    Give an algebra presentation for $H^\ast(\widetilde{\Poly}^D_n(\Rb^{N});\Zb)$.
\end{problem}
We do not solve Problem \ref{problem:presentation} here, but only give a set of algebra generators. A solution to Problem~\ref{problem:presentation} would shed significant light on the algebra structure of the $E_2$-page of the Leray spectral sequence for the inclusion $\widetilde{\Poly}^D_n(X)\to X^D$.

Before continuing we will need to make a definition.

\begin{definition}[{\boldmath$\cd(x),\coor(x)$}]
 When $x$ is an element of the intersection lattice of a subspace arrangement over $\Rb$,  we will denote by $\cd(x)$ the codimension of the subspace $x$, and when $x$ is also a complex subspace, we write $\cd_\C(x)$ for its complex codimension.  Let
 \[\coor(x):= H^{N}_c(\Rb^N;\Z)\tensor  \Hom(H^{\dim(x)}_c(x;\Z),\Z).\]
 More generally, given a smooth closed submanifold $Z$ in a manifold $X$, define \[\coor(Z):=H^{\dim(X)}_c(X;\Z)\otimes\Hom(H^{\dim(Z)}_c(Z;\Z),\Z).\]
\end{definition}

We will need the following form of the Goresky-MacPherson formula. The first statement follows from Deligne--Goresky--MacPherson \cite[Corollary 1.8, and \S 1.10-1.11]{DGM}, and the fact that the morphism constructed in \cite[\S 1.6]{DGM} is clearly equivariant under our hypotheses. For the second statement, see Deligne--Goresky--MacPherson \cite[\S 4.2]{DGM}; or de Longueville--Schultz \cite[Theorem 5.2]{LS}, for a related treatment.

\begin{theorem}[{\bf Goresky-MacPherson Formula}]\label{theorem:GM}
    Let $\Ac:=\{L_i\}$ be an arrangement of linear subspaces in $\Rb^N$, and let $\Pi_{\Ac}$ denote its intersection lattice.  Let $\Mc_{\Ac}:=\Rb^N-\bigcup_i L_i$.

Suppose that  for every $x,y \in \Pi_{\Ac}$ with $x<y$ we have  $\cd(y)-\cd(x)\ge 2$, and
  $H^*(\Mc_{\Ac};\Zb)$ is a free $\Z$-module.  Then we have the following.
    \begin{enumerate}
        \item There exists an isomorphism of abelian groups
            \begin{equation}\label{GM}
                H^i(\Mc_{\Ac};\Zb)\cong \bigoplus_{x\in \Pi_{\Ac}} \tilde{H}_{\cd(x)-i-2}(\Delta(\overline{\Pi_{\Ac}(\le x)});\Zb) \tensor \coor(x)
            \end{equation}
            that is equivariant with respect to invertible linear maps $\sigma\in\GL(\Rb,N)$ that  preserve the arrangement $\Ac$.  (The action on the left-hand side comes from the action of the linear maps on $\Mc_{\Ac}$ and the action on the right hand side comes from the induced action on $\Pi_{\Ac}$ and the induced actions on the
     $\coor(x)$.)
        \item Taking the direct sum over all $i\geq 0$ in \eqref{GM}  gives an isomorphism of graded-commutative algebras :
          \begin{equation}\label{GM4}
                H^*(\Mc_{\Ac};\Zb)\cong \bigoplus_i \bigoplus_{x\in \Pi_{\Ac}} \tilde{H}_{\cd(x)-i-2}(\Delta(\overline{\Pi_{\Ac}(\le x)});\Zb)\otimes\coor(x)
            \end{equation}
where the algebra structure on the right-hand side of \eqref{GM4} is the intersection product \eqref{intprod} on the tensor factors for the order complex, and is given by the natural maps $\coor(x)\otimes\coor(y)\to\coor(x\cap y)$, which are isomorphisms for subspaces $x$ and $y$ which intersect transversely, and 0 otherwise.
    \end{enumerate}
\end{theorem}

Now let $I$ be an $n$-equals partition of the colored set $D$, with blocks $I_1,\ldots ,I_e$.
Let $X_I\subset X^D$ be the subset where coordinates from the same $I_i$ are all equal.  Note that $\dim X_I=e\dim (X)$ and $X_I\cong \prod_i X_{I_i}$ where $X_{I_i}\cong X$ and the isomorphism records the equality of the coordinates indexed by $I_i$. Define
\[
    \cd(I,X):=\dim_{\Rb}(X)\cdot (|D|-e),
\]
which is the codimension of $X_I$ in $X^D$.
When $X=\Rb^N$, the $X_I$ form a linear subspace arrangement, and we have the map of posets $\Pi^D_n \ra \Pi_{\{X_I\}_I}$ given by $I\mapsto X_I$ is an isomorphism.  Further, we have an inclusion of the colored $n$-equals arrangement $\mathcal{A}^{N,D}_n \ra \{X_I\}_I$, and it is not hard to see that they have the same intersection lattices.  So, we also have an isomorphism of posets $\Pi^D_n \ra  \Pi^D_n(\Rb^N)$ taking $I$ to $X_I$.

\begin{definition}[{\boldmath$\epsilon_I(q)$}]
    Let $I$ be an $n$-equals partition of the colored set $D$, with blocks $I_1,\ldots,I_e$.  Define $\epsilon_I(q)$ to be the $\stab_I$-equivariant constant sheaf, supported on $X_I$, whose stalk at each point equals
    \[\tilde{H}_{\cd(I,X)-q-2}(\Delta(\overline{\Pi_n^D(\leq I)});\Z)\otimes\coor(X_I)\iso \bigoplus _{k_1+\cdots +k_e=\cd(I,X)-q-2e}\bigboxtimes_{i=1}^e \widetilde{H}_{k_i}(\Delta(\overline{\Pi^{I_i}_n});\Z)\otimes\coor(X_{I_i})\]
    where $\bigboxtimes$ denotes the external tensor product of the constant sheaves $\widetilde{H}_{k_i}(\Delta(\overline{\Pi^{I_i}_n});\Z)\otimes\coor(X_{I_i})$ on $X_{I_i}$ and where $\stab_I$ acts on the tensor factors by
    \begin{equation*}
        \tilde{H}_{k_i}(\Delta(\overline{\Pi_n^{I_i}});\Z)\otimes\coor(X_{I_i})\to^{\sigma_*} \tilde{H}_{k_i}(\Delta(\overline{\Pi_n^{\sigma\cdot I_i}});\Z)\otimes\coor(X_{\sigma\cdot I_i}).
    \end{equation*}
\end{definition}

Theorem \ref{theorem:GM} for colored $n$-equals arrangement $\mathcal{A}^{N,D}_n$ endows the direct sum $\bigoplus_q \bigoplus_{I\in \Pi^D_n}\epsilon_I(q)$ with a canonical structure as a sheaf of graded algebras.

\begin{theorem}[{\bf The Leray spectral sequence}]\label{theorem:leray}
    Let $X$ be a connected, orientable manifold of dimension $N\ge 2$. Let $\tilde{E}_2^{p,q}(X,D,n)$ denote the $(p,q)$ term of the $E_2$-page of the Leray spectral sequence for the inclusion $\pi: \widetilde{\Poly}^D_n(X)\ra X^D$,  computing the cohomology of the constant sheaf $\Z$ on $\widetilde{\Poly}^D_n(X)$.  There exists an $S_D$-equivariant isomorphism of bigraded algebras
    \[
        \bigoplus_{p,q}\tilde{E}_2^{p,q}(X,D,n)\iso \bigoplus_{p,q}\bigoplus_{I\in\Pi^D_n}H^p(X_I;\epsilon_I(q)).\]

        Further, when $X$ is a smooth, complex algebraic variety, this isomorphism respects mixed Hodge structures.  Here the mixed Hodge structure on \[\epsilon_I(q)\iso \tilde{H}_{\cd(I,X)-q-2}(\Delta(\overline{\Pi_n^D(\leq I)});\Z)\otimes\coor(X_I)\] is trivial on the first tensor factor and the canonical one (i.e.\ pure of type $(\cd_\C(X_I),\cd_\C(X_I))$) on the second tensor factor.
        \end{theorem}

\begin{proof}[Proof of Theorem \ref{theorem:leray}]
    By the definition of the Leray spectral sequence, the theorem reduces to showing that, given the inclusion $\pi: \widetilde{\Poly}^D_n(X)\ra X^D$, there is an $S_D$-equivariant isomorphism of sheaves of graded algebras:
    $$
        \bigoplus_q R^q \pi_* \Z \cong \bigoplus_q \bigoplus_{I\in  \Pi^D_n} \epsilon_I(q).
    $$
    where the $S_D$ action on the right-hand side is given on the underlying spaces by $\sigma:X_I\to X_{\sigma\cdot I}$ and the map of sheaves $\epsilon_I(q)\to\sigma^*\epsilon_{\sigma\cdot I}(q)$ is given on stalks by
    \[ \epsilon_I(q)_x\cong \tilde{H}_{\cd(I,X)-q-2}(\Delta(\overline{\Pi_n^D(\leq I)});\Z)\to^{\sigma_*} \tilde{H}_{\cd(\sigma\cdot I,X)-q-2}(\Delta(\overline{\Pi_n^D(\leq \sigma\cdot I)});\Z)\cong \epsilon_{\sigma\cdot I}(q)_{\sigma\cdot x}\]
    where the first and last isomorphisms are those of Proposition \ref{proposition:kunneth}.

    For each $q$, we will give an $S_D$-equivariant map of sheaves
    $$
        E:=\bigoplus_{I\in  \Pi^D_n} \epsilon_I(q) \ra  R^q \pi_* \Z.
    $$
  We give the map of sheaves by giving it on the basis of open sets on $X^D$ consisting of all sets of the form $U=U_1 \times U_2 \times \cdots$ where each $U_j$ is a small, nice contractible open and such that for each $j,k$ either $U_j=U_k$ or $U_j\cap U_k=\emptyset$.

    To such a $U$ we can associate a partition $J$ of $D$ according to which $U_i$ are equal.  Then, as in \cite[Proof of Theorem 1]{To} we have
    \begin{equation}\label{eq:product1}
        H^q(U \cap \widetilde{\Poly}^D_n(X);\Z )\cong H^q(\widetilde{\Poly}^{J_1}_n(\Rb^{N}) \times \widetilde{\Poly}^{J_2}_n(\Rb^{N}) \times \cdots
        ;\Zb).
    \end{equation}
    By the K\"unneth isomorphism, we have
    \begin{equation}
        \label{eq:kun1}
          H^q(\widetilde{\Poly}^{J_1}_n(\Rb^{N}) \times \widetilde{\Poly}^{J_2}_n(\Rb^{N}) \times \cdots
        ;\Zb)\iso  \bigoplus_{i_1+\cdots +i_\ell=q}\bigotimes_{a=1}^{\ell}H^{i_a}(\widetilde{\Poly}^{J_a}_n(\Rb^{N});\Zb).
    \end{equation}

For $N\ge 2$, the colored $n$-equals arrangement $\mathcal{A}^{N,D}_n$  is easily checked to satisfy the  codimension  assumptions of the Goresky-MacPherson formula
(Theorem \ref{theorem:GM}), and thus, for each $i\geq 0$, there is an $S_D$-equivariant isomorphism of graded algebras:
    \begin{align}
     \label{eq:GMplusBW}
        \bigoplus_{i_a} H^{i_a}(\widetilde{\Poly}^{J_a}_n(\Rb^{N});\Zb)\cong \bigoplus_{i_a} \bigoplus_{I_a\in\Pi^{J_a}_n}
        \widetilde{H}_{\cd(I_a,X)-i_a-2}(\Delta(\overline{\Pi_n^{J_a}(\leq I_a)});\Z)\tensor \coor(X_{I_a})
    \end{align}
    Plugging this in to \eqref{eq:kun1} and distributing terms yields an isomorphism of graded algebras:

    \begin{equation}\label{eq:kun2}
        \bigoplus_q  H^q(U \cap \widetilde{\Poly}^D_n(X);\Z )\iso \bigoplus_q \bigoplus_{(I_1,\ldots,I_\ell)\in\Pi^{J_1}_n\times\cdots\times\Pi^{J_\ell}_n}\bigoplus_{i_1+\cdots +i_\ell=q}\bigotimes_{a=1}^\ell\tilde{H}_{\cd(I_a,X)-i_a-2}(\Delta(\overline{\Pi_n^{J_a}(\leq I_a)});\Z)\tensor \coor(X_{I_a})
    \end{equation}

    By Proposition \ref{proposition:kunneth}, the right-hand side of \eqref{eq:kun2} is isomorphic as a graded algebra to :

    $$
        \bigoplus_q\bigoplus_{I\in  \Pi^D_n(\leq J)} \tilde{H}_{\cd(I,X)-q-2}(\Delta(\overline{\Pi_n^D(\leq I)});\Z)\tensor \coor(X_{I})
    $$
    which 
    is isomorphic as a graded algebra to $\bigoplus_q\bigoplus_{I\in\Pi^D_n}\epsilon_I(q)(U)$.

If $V$ is another open set in our basis, with associated partition $K$, then if there is an inclusion $V\to U$, then $K\leq J$.
The map $V\to U$ induces a homomorphism
    \begin{align}\label{eq:restriction:morphism}
        H^q(U \cap \widetilde{\Poly}^D_n(X);\Z)&\to H^q(V \cap \widetilde{\Poly}^D_n(X);\Z )\\
 \bigoplus_{I\in  \Pi^D_n(\leq J)} \tilde{H}_{\cd(I,X)-q-2}(\Delta(\overline{\Pi_n^D(\leq I)});\Z\tensor \coor(X_{I}) &\to \notag
  \bigoplus_{I\in  \Pi^D_n(\leq K)} \tilde{H}_{\cd(I,X)-q-2}(\Delta(\overline{\Pi_n^D(\leq I)});\Z)\tensor \coor(X_{I})
        .
    \end{align}
    From \cite[Equation 1.11.3]{DGM}, we see that this morphism is just given by sending each summand where $I\not\leq K$ to $0$, and is the identity map on summands where $I \leq K$.  This agrees with the restriction map $\bigoplus_{I\in\Pi^D_n}\epsilon_I(q)(U)\to \bigoplus_{I\in\Pi^D_n}\epsilon_I(q)(V)$, so gives an isomorphism of sheaves of algebras.

We now prove the second statement of the theorem.  Our computations above, combined with
 Proposition~\ref{proposition:kunneth}, give that the $E_2$ page of the Leray spectral sequence is generated by $H^*(X^D)$ together with $\bigoplus_{I'}H^*(X_{I'};\epsilon_{I'}(q))$, where the direct sum is over partitions $I'\in \Pi^D_n$  with precisely one nonsingleton block.   It is therefore enough to compute the weights on these summands.

 The summand $H^*(X^D)$ is just the restriction along the inclusion $\widetilde{\Poly}^D_n(X)\into X^D$, and so carries its canonical mixed Hodge structure, as claimed.  Now consider each of the other summands $H^*(X_{I'};\epsilon_{I'}(q))$.  Write $I'$ as $I_1$ plus singletons.  There is a commutative diagram of varieties:

 \[
 \begin{array}{lll}
 \widetilde{\Poly}^D_n(X)&\into& X^D\\
 \downarrow& &\downarrow\\
  \widetilde{\Poly}^{I_1}_n(X)&\into_j& X^{I_1}\\
 \end{array}
 \]
 By the naturality of the Leray spectral sequence, this commutative diagram induces a map of Leray spectral sequences from that of the bottom row to that of the top row.   By the first part of the theorem, the image of this map equals $\bigoplus_{J\leq I'}H^*(X_{J};\epsilon_{J}(q))$.   It thus suffices to compute the Hodge type for the $E_2$ page of the Leray spectral sequence for the inclusion $j_{X,I_1}\colon \widetilde{\Poly}^{I_1}_n(X)\into X^{I_1}$.  This will follow from a standard argument, which we now recall for the sake of the reader.

 For simplicity of notation, let $W:=\C^{\dim_\C(X)}$.  Recall that for a colored, $n$-equals partition
 $J\in\Pi^{I_1}_n$, we denote by $W_J$ be the linear subspace in $W^{I_1}$ defined by setting coordinates to be equal as determined by the partition $J$.  For $x\in X^{I_1}-\widetilde{\Poly}^{I_1}_n(X)$, a standard argument using Noether Normalization shows that there is an \'{e}tale neighborhood $f:U\to X^{I_1}$ of $x$ in $X^{I_1}$ that admits an \'{e}tale map $\pi:U\to W^{I_1}$ with the property that $\pi^{-1}(W_J)=f^{-1}(X_J)$ for each $J\in \Pi^{I_1}_n$.
 So for $x\in X_J$ we have an isomorphism
 \[(R^q{j_{X,I_1}}_*\Z)_x\to^\cong R^q(j_{W,J})_* \Z)_{y}.\]
 where the right hand side denotes the stalk at a generic $y\in W_J$ of the push-forward of the constant sheaf along the inclusion \[j_{W,J}\colon\widetilde{\Poly}^{J_1}_n(W)\times W^{I_1\setminus J_1}\into W^{I_1}.\]  Now $\widetilde{\Poly}^{J_1}_n(W)\times W^{I_1\setminus J_1}$ is just a linear subspace complement, and so the Hodge types are known by work of Bj\"orner-Ekedahl \cite[Theorem 4.9]{BE}.\footnote{Theorem 4.9 of \cite{BE} is stated for \'etale cohomology of arrangements over finite fields. However, as observed in the last line of p. 168 of \emph{loc. cit.}, the arguments give the analogous statement for mixed Hodge structures of the cohomology of arrangements over $\C$.  This is also directly addressed in Example 1.14 of \cite{DGM}.} Because $W^{I_1}\cong(\Ab^r)^{I_1}$ is acyclic, the $E_2$ page of the Leray spectral sequence for $j_{W,J}$ is precisely $H^0$ of the pushforward sheaves $R^q(j_{W,j})_\ast \Z$.  Further, the work of Bj\"orner-Ekedahl identifies the $E_2$-page with the cohomology of the arrangement.  Taken together, this gives that
 \begin{equation*}
    (R^q(j_{X,I_1})_\ast \Z)_x\cong \bigoplus_{K\le J}H_{\cd(K,X)-q-2}(\Delta(\overline{\Pi^{I_1}_n(\le K)});\Z)
    \otimes\coor(W_K).
 \end{equation*}
 which in turn equals the stalk at $x$ of the sheaf $\bigoplus_{K\leq I'}\epsilon_K(q)$; see Theorem 0.1 of \cite{Sa} for the fact that this sheaf is a mixed Hodge module, and therefore its stalks are endowed with mixed Hodge structures.   Note that the codimension of $W_K$ in $W^{I_1}$ equals the codimension of
 $X_{K}$ in $X^{D}$.   Combined with the above, this gives an isomorphism of sheaves
 \[R^q_{j_\ast}\Z\cong\bigoplus_{K\leq I'}\epsilon_K(q)\]
 This completes the proof of the second statement of the theorem.
 \end{proof}

\section{The Local Computation: \texorpdfstring{$X=\Rb^{N}$}{Affine Space}}
\label{section:local}

Given a colored set $D$, define
\begin{equation*}
    \bd:=(|D(1)|,\ldots,|D(m)|)\in\Nb^m
\end{equation*}
where $|D(i)|$ denotes the number of elements of $D$ of color $i$. The permutation action of $S_D$ on $X^D$ leaves invariant $\widetilde{\Poly}^D_n(X)$.  We denote the quotient space by
\[\Poly^{\bd}_n(X):=\widetilde{\Poly}^D_n(X)/S_D.  \]

The goal of this section is to compute $H^i(\Poly_n^{\bd}(\Rb^{N});\Qb)$ for $N\ge 2$.    If we had an explicit presentation of the algebra $H^*(\widetilde{\Poly}_n^{\bd}(\Rb^{N});\Qb)$, one might hope to compute the $S_D$-invariants directly,  thus giving $H^i(\Poly_n^{\bd}(\Rb^{N});\Qb)$ by transfer.
We do not know such a presentation.  Instead, we induct on a canonical filtration of $\Sym^{\bd}(\Rb^{N})$, extending arguments in \cite{FW}.  The method goes back to Arnol'd \cite{Ar} , and Segal \cite{Se}.

Given $\ell$, let $\bd+\ell$ denote the vector $(d_1+\ell,\ldots,d_m+\ell)$, let $\bd+\ell_i$ denote $(d_1,\ldots,d_i+\ell,\ldots,d_m)$, and let $\ell\cdot\bd$ denote the vector $(\ell d_1,\ldots,\ell d_m)$.  Our goal in this section is to prove the following.

\begin{theorem}[{\bf Local computation}]
\label{theorem:local}
    Fix $r\ge 1$, $n$, $m$ and $\bd$ with $d_i\ge n$ for all $i$.
    \begin{enumerate}
        \item If $N=2r+1$, then
            \begin{equation}
                \label{eq:bettinums2od}
                H^i(\Poly_n^{\bd}(\Rb^{2r+1});\Qb)\cong
                    \left\{
                    \begin{array}{ll}
                        \Qb & i=0\\
                        0 & \text{else}
                    \end{array}\right.
            \end{equation}
        \item If $N=2r$, then
           \begin{equation}
            \label{eq:bettinums2ev}
                H^i(\Poly_n^{\bd}(\Rb^{2r});\Qb)\cong
                    \left\{
                    \begin{array}{ll}
                        \Qb & i=2r(mn-1)-1\\
                        \Qb & i=0\\
                        0 & \text{else}
                    \end{array}\right.
            \end{equation}
    \end{enumerate}
\end{theorem}

\subsection{Proof of Theorem~\ref{theorem:local}: the top level}
    Let $\Qb_{\rm or}$ denote the orientation sheaf on $\widetilde{\Poly}^D_n(\Rb^N)$. Because $\widetilde{\Poly^D}_n(\Rb^N)$ is an oriented manifold, transfer followed by Poincar\`e duality gives:
    \begin{align*}
        H^\ast(\Poly^{\bd}_n(\Rb^N);\Qb)&\cong H^\ast(\widetilde{\Poly}^D_n(\Rb^N);\Qb)^{S_D}\\
        &\cong (H^{N|D|-\ast}_c(\widetilde{\Poly}^D_n(\Rb^N);\Qb)\otimes\Qb_{\rm or})^{S_D}.
    \end{align*}
    The $S_D$ action on $\Qb_{\rm or}$ is given by
    \begin{equation*}
        \Qb_{or}:=\left\{
            \begin{array}{cc}
                \Qb_{\rm sgn} & N\text{ odd}\\
                \Qb_{\rm triv} & N\text{ even}
            \end{array}\right.
    \end{equation*}
    where $\Qb_{\rm sgn}$ and $\Qb_{\rm triv}$ are the restrictions to $S_D\subset S_{|D|}$ of the sign and trivial representations.  This is the critical place where the even and odd dimensional cases differ.

\subsection*{Case 1: $N>1$ odd.}

     Let $A_D:=S_D\cap A_{|D|}\subset S_{|D|}$. By transfer and the fact that $S_D$ acts on $\Qb_{or}$ by the sign representation, we have that
    \begin{equation}\label{eq:alt}
        (H^\ast_c(\widetilde{\Poly}^D_n(\Rb^N);\Qb)\otimes\Qb_{\rm or})^{S_D}\cong (H^\ast_c(\widetilde{\Poly}^D_n(\Rb^N)/A_{D};\Qb)\otimes\Qb_{\rm or})^{S_2}.
    \end{equation}
    Define $\widetilde{R}^D_{n,1}(\Rb^N):=(\Rb^N)^D-\widetilde{\Poly}^D_n(\Rb^N)$. The open embedding
    \begin{equation*}
        \widetilde{\Poly}^D_n(\Rb^N)\to (\Rb^N)^D
    \end{equation*}
    gives rise to a long exact sequence in compactly supported cohomology
    \begin{equation*}
        \cdots\to H^i_c(\widetilde{\Poly}^D_n(\Rb^N);\Qb)\to H^i_c((\Rb^N)^D;\Qb)\to H^i_c(\widetilde{R}^D_{n,1}(\Rb^N);\Qb)\to H_c^{i+1}(\widetilde{\Poly}^D_n(\Rb^N);\Qb)\to\cdots
    \end{equation*}
    By Equation \eqref{eq:alt} and Poincar\`e duality, the theorem for $N>1$ odd is equivalent to :
    \[(H^\ast_c(\widetilde{R}^D_{n,1}(\Rb^N)/A_D;\Qb)\otimes\Qb_{\rm sgn})^{S_2}=0.\]

To see this, it suffices to observe that $S_2$ acts trivially on the space $\widetilde{R}^D_{n,1}(\Rb^N)/A_D$. Indeed, by definition, any $x\in \widetilde{R}^D_{n,1}(\Rb^N)$ has at least two coordinates, say $a_i$ and $b_i$,  of each color $i$ being equal. On the other hand, the $S_2$-action on the orbit space $\widetilde{R}^D_{n,1}(\Rb^N)/A_D$ is given by applying \emph{any} transposition to an orbit representative. Picking the transposition $(a_1~b_1)$, we see that $S_2\cdot [x]=[x]$ as claimed. This completes the proof of Case 1.

 \subsection*{Case 2: $N>1$ even.}

 Now let $N=2r$. We will make repeated use of the fact that for any $k$,
    \begin{equation*}
        H^\ast(\Sym^k(\Rb^{2r});\Qb)=\left\{
        \begin{array}{lr}
            \Qb & \ast=0\\
            0 & \text{else}
        \end{array}\right.
    \end{equation*}
    by transfer, and, similarly, for any space $X$,
    \begin{equation*}
        H^\ast_c(X\times\Sym^k(\Rb^{2r});\Qb)\cong H^{\ast-2rk}_c(X;\Qb).
    \end{equation*}
    Now, as observed above, by Poincar\'{e} duality and transfer it suffices to prove the version in compactly supported cohomology.

    Our argument follows the lines of the argument in \cite{FW}, which itself is an extension of the arguments in Segal \cite{Se}.  Recall the filtration:
    \begin{equation*}
        \Rb^{2r|\bd|}=R_{n,0}^{\bd}(\Rb^{2r})\supset R_{n,1}^{\bd}(\Rb^{2r})\supset\cdots \supset \emptyset.
    \end{equation*}
    with $R^{\bd}_{n,k}(\Rb^{2r})$ the space of $m$-tuples $(D_1,\ldots,D_m)$ of effective 0-cycles on $\Rb^{2r}$, with $\deg(D_i)=d_i$, for which there exists an effective $0$-cycle $D$ of degree at least $k$, and effective $0$-cycles $C_i$, such that $D_i=C_i+nD$ for each $i$.    By the same arguments as in \cite{FW}, there are homeomorphisms
    \begin{equation*}
        R^{\bd}_{n,k}(\Rb^{2r})-R^{\bd}_{n,k+1}(\Rb^{2r})\cong\Poly^{\bd-nk}_n(\Rb^{2r})\times \Sym^k(\Rb^{2r}).
    \end{equation*}
    Since $N=2r$ and since $\Qb_{\rm or}\cong \Qb_{\rm triv}$ as $S_D$-equivariant sheaves, to prove the theorem it is enough to prove Equation \eqref{eq:cbettinums2} in the following.
    \begin{proposition}\label{proposition:zcoho1}
        Let $r\geq 1$. Let $\bd\ge n\cdot\vec{1}$ and $k\le \lfloor \frac{\min_i d_i}{n}\rfloor$. Then
        \begin{equation}
            \label{eq:cbettinums2}
                H^i_c(\Poly_n^{\bd}(\Rb^{2r});\Qb)\cong
                \left\{
                    \begin{array}{ll}
                        \Qb & i=2r(|\bd|-mn+1)+1\\
                        \Qb & i=2r|\bd|\\
                        0 & \text{else}
                    \end{array}\right.
        \end{equation}
        \begin{equation*}
            H^i_c(R_{n,k}^{\bd}(\Rb^{2r});\Qb)\cong
                \left\{
                    \begin{array}{ll}
                        \Qb & i=2r(|\bd|-k(mn-1))\\
                        0 & \text{else}
                    \end{array}\right.
        \end{equation*}
        and for each $j\geq 1$ the continuous open embeddings
        \begin{align}\label{stabS}
            \Poly_n^{\bd}(\Rb^{2r})\times\Rb^{2r}&\to\Poly_n^{\bd+1_j}(\Rb^{2r})\\
            \label{stabF}
             R_{n,k}^{\bd}(\Rb^{2r})\times\Rb^{2r}&\to R_{n,k}^{\bd+1_j}(\Rb^{2r})
       \end{align}
        given by ``bringing zeroes in from infinity'' induce isomorphisms on compactly supported rational cohomology.
    \end{proposition}
    Our goal in the rest of this section is to prove Proposition~\ref{proposition:zcoho1}, and thus
    Theorem~\ref{theorem:local}.

  \subsection{Proof of Proposition~\ref{proposition:zcoho1}}
  We prove the proposition by induction on $(\bd,k)$, ordered lexicographically upward on the entries in $\bd$ and downward on $k$.

  For the base case $\bd=n\cdot\vec{1}:=(n,\ldots,n)$, the isomorphism \eqref{eq:cbettinums2} follows immediately from the isomorphism
  \begin{equation*}
    \Poly_n^{n\cdot\vec{1}}(\Rb^{2r})\cong\Sym^{\vec{n}}(\Rb^{2r})-\Rb^{2r}\cong \Sym^n(\Rb^{2r})^m-\Rb^{2r}.
  \end{equation*}
  Similarly, we have isomorphisms
    \begin{align*}
        R_{n,1}^{n\cdot\vec{1}}&\cong\Rb^{2r}\\
        R_{n,0}^{n\cdot\vec{1}}&\cong\Sym^{n\cdot\vec{1}}(\Rb^{2r})\cong\Sym^n(\Rb^{2r})^m
    \end{align*}
  and a map of cofiber sequences
    \begin{equation*}
        \xymatrix{
            R_{n,1}^{n\cdot\vec{1}+1_i}(\Rb^{2r})^+ \ar[r] \ar[d] & R_{n,0}^{n\cdot\vec{1}+1_i}(\Rb^{2r})^+ \ar[r] \ar[d] & (\Poly_n^{n\cdot\vec{1}+1_i}(\Rb^{2r}))^+ \ar[d]\\
                (R_{n,1}^{n\cdot\vec{1}}(\Rb^{2r})\times\Rb^{2r})^+ \ar[r] & (R_{n,0}^{n\cdot\vec{1}}(\Rb^{2r})\times\Rb^{2r})^+ \ar[r] & (\Poly_n^{n\cdot\vec{1}}(\Rb^{2r})\times\Rb^{2r})^+
            }
    \end{equation*}
    where $X^+$ denotes the 1-point compactification of $X$. This is homeomorphic to
    \begin{equation*}
            \xymatrix{
                (\Rb^{4r})^+ \ar[r] \ar[d] & (\Sym^{n\cdot\vec{1}+1_i}\Rb^{2r})^+ \ar[r] \ar[d] & (\Poly_n^{n\cdot\vec{1}+1_i}(\Rb^{2r}))^+ \ar[d]\\
                (\Rb^{2r}\times\Rb^{2r})^+ \ar[r] & (\Sym^{n\cdot\vec{1}}(\Rb^{2r})\times\Rb^{2r})^+ \ar[r] & (\Poly_n^{n\cdot\vec{1}}(\Rb^{2r})\times\Rb^{2r})^+
            }
    \end{equation*}
    Because the first two vertical maps induce isomorphisms in cohomology, the Five Lemma (applied to the map of long exact sequences in cohomology) shows that the right vertical map induces a cohomology isomorphism as well.  This establishes the base case of the induction.

    Given these base cases, the main claim \eqref{eq:cbettinums2} of the proposition follows from the claim that \eqref{stabS} is an isomorphism. Suppose we have proved the isomorphism \eqref{stabF} for all $(\bd',k)$. Consider the map of cofiber sequences
    \begin{equation*}
        \xymatrix{
            R_{n,1}^{\bd+1_i}(\Rb^{2r})^+ \ar[r] \ar[d] & (\Sym^{\bd+1_i}(\Rb^{2r}))^+ \ar[r] \ar[d] & (\Poly_n^{\bd+1_i}(\Rb^{2r}))^+ \ar[d]\\
            (R_{n,1}^{\bd}(\Rb^{2r})\times\Rb^{2r})^+ \ar[r] & (\Sym^{\bd}(\Rb^{2r})\times\Rb^{2r})^+ \ar[r] & (\Poly_n^{\bd}(\Rb^{2r})\times\Rb^{2r})^+
        }
    \end{equation*}
    and note that the center column is just the case of \eqref{stabF} for $k=0$. Our assumption on \eqref{stabF} gives that the left-hand and central vertical maps induce isomorphisms on cohomology.  Applying the Five Lemma to the long exact sequence in cohomology gives that the right-hand vertical map induces an isomorphism on cohomology, proving the isomorphism \eqref{stabS}.

    It remains to prove the formula for the compactly supported cohomology of $R^{\bd}_{n,k}(\Rb^{2r})$ and the isomorphism \eqref{stabF}. We have already shown the base case. For the inductive step, suppose now that we have shown the proposition for all $(\bd',k)$ with $\bd'<\bd$. We will show, by downward induction on $k$, that the proposition holds for $(\bd,k)$ for all $k$. To prove this, we consider two cases.

    \para{Case 1: \boldmath$n\nmid \bd+1_i$} This case means that $\bd+1_i\neq n\cdot\bd'$ for some $\bd'$. For the base case of the downward induction on $k$, i.e. $k=\lfloor\frac{\min_i d_i}{n}\rfloor$, a similar observation to the above shows that the map \eqref{stabF} induces an isomorphism on compactly supported cohomology.

    Now suppose we have shown that \eqref{stabF} induces such an isomorphism for $\bd+1_i$ and $k+1>1$. Observe that the ``bringing in zeroes'' maps fit together to give a continuous map of cofiber sequences
    \begin{equation*}
        \xymatrix{
                R_{n,k+1}^{\bd+1_i}(\Rb^{2r})^+ \ar[r] \ar[d] & R_{n,k}^{\bd+1_i}(\Rb^{2r})^+ \ar[r] \ar[d] & (\Poly_n^{(\bd+1_i)-kn}(\Rb^{2r})\times\Sym^k(\Rb^{2r}))^+ \ar[d]\\
                (R_{n,k+1}^{\bd}(\Rb^{2r})\times\Rb^{2r})^+ \ar[r] & (R_{n,k}^{\bd}(\Rb^{2r})\times\Rb^{2r})^+ \ar[r] & (\Poly_n^{\bd-kn}(\Rb^{2r})\times\Sym^k(\Rb^{2r})\times\Rb^{2r})^+
            }
    \end{equation*}
    Our inductive hypotheses show that the left and right vertical maps induce isomorphisms on cohomology. Applying the Five Lemma to the long exact sequences in cohomology gives that the central map is an isomorphism in cohomology. This concludes the induction step, and thus the proof, when $n\nmid \bd+1_i$.

    \para{Case 2: \boldmath$\bd+1_i=an\cdot\vec{1}$ for \boldmath$a>1$}  In this case, the induction proceeds as above, once we establish the cases $k=a$ and $k=a-1$. The claim about the cohomology of $R_{n,a}^{an\cdot\vec{1}}$ follows from the isomorphism
    \begin{equation*}
        R_{n,a}^{an\cdot\vec{1}}\cong \Sym^a(\Rb^{2r}).
    \end{equation*}
    For $k=a-1$, the identification
    \begin{equation*}
        R_{n,a-1}^{an\cdot\vec{1}}-R_{n,a}^{an\cdot\vec{1}}\cong \Poly_n^{n\cdot\vec{1}}(\Rb^{2r})\times\Sym^{a-1}(\Rb^{2r})\cong (\Sym^{n\cdot\vec{1}}(\Rb^{2r})-\Rb^{2r})\times\Sym^{a-1}(\Rb^{2r})
    \end{equation*}
    gives rise to the long exact sequence in compactly supported cohomology
        \begin{equation*}
            \cdots\to H^{p-2r(a-1)}_c(\Sym^{n\cdot\vec{1}}(\Rb^{2r})-\Rb^{2r};\Qb)\to H^p_c(R_{n,a-1}^{an\cdot\vec{1}}(\Rb^{2r});\Qb)\to H^p_c(\Sym^a(\Rb^{2r});\Qb)\to^\partial\cdots
        \end{equation*}
        This implies that
        \begin{equation*}
            H^p_c(R_{n,a-1}^{an\cdot\vec{1}}(\Rb^{2r});\Qb)\cong
                \left\{
                    \begin{array}{ll}
                        0 & p<2ra\\
                        0 & 2ra+1<p<2r(mn+a-1)\\
                        \Qb & p=2r(mn+a-1)\\
                        0 & p>2r(mn+a-1)
                    \end{array}\right.
        \end{equation*}
        This leaves the cases $p=2ra$ and $p=2ra+1$. For these, we have a long exact sequence
        \begin{align*}
            0\to H^{2ra}_c(R_{n,a-1}^{an\cdot\vec{1}}(\Rb^{2r});\Qb)\to H^{2ra}_c(\Rb^{2ra};\Qb)\to^\partial &H^{2ra+1}_c((\Sym^{n\cdot\vec{1}}(\Rb^{2r})-\Rb^{2r})\times\Sym^{a-1}(\Rb^{2r});\Qb)\\
            &\to H^{2ra+1}_c(R_{n,a-1}^{an\cdot\vec{1}}(\Rb^{2r});\Qb)\to 0.
        \end{align*}
        It suffices to show that the boundary map is an isomorphism. To see this, consider the closed embedding
        \begin{align*}
            \Rb^{2r}\times\Sym^{a-1}(\Rb^{2r})&\to \Sym^{n\cdot\vec{1}}(\Rb^{2r})\times\Sym^{a-1}(\Rb^{2r})\\
            (\bar{z},D)&\mapsto(n\cdot\bar{z},\cdots,n\cdot\bar{z},D)
        \end{align*}
        where we view $\Rb^{2r}\times\Sym^{a-1}(\Rb^{2r})$ as the variety of pairs of effective 0-cycles $(\bar{z},D)$ on $\Cb^{r}$ with $\deg(\bar{z})=1$ and $\deg(D)=a-1$, and where we view
        $\Sym^{n\cdot\vec{1}}(\Rb^{2r})\times\Sym^{a-1}(\Rb^{2r})$ as the variety of $(m+1)$-tuples of effective 0-cycles
        \begin{equation*}
            (D_1,\cdots,D_m,D)
        \end{equation*}
        with $\deg(D_i)=n$ and $\deg(D)=a-1$. By inspection,
        \begin{equation*}
            \Sym^{n\cdot\vec{1}}(\Rb^{2r})\times\Sym^{a-1}(\Rb^{2r})-\Rb^{2r}\times\Sym^{a-1}(\Rb^{2r})\cong\Poly_n^{n\cdot\vec{1}}(\Rb^{2r})\times\Sym^{a-1}(\Rb^{2r})
        \end{equation*}
        and the assignments
        \begin{align*}
            (\bar{z},D)&\mapsto (\bar{z}+D)\\
            (D_1,\cdots,D_m,D)&\mapsto(D_1+nD,\cdots,D_m+nD)
        \end{align*}
        determine a map of cofiber sequences
        \begin{equation*}
            \xymatrix{
                (\Rb^{2r}\times\Sym^{a-1}(\Rb^{2r}))^+ \ar[r] \ar[d] & (\Sym^{n\cdot\vec{1}}(\Rb^{2r})\times\Sym^{a-1}(\Rb^{2r}))^+ \ar[r] \ar[d] & (\Poly_n^{n\cdot\vec{1}}(\Rb^{2r})\times\Sym^{a-1}(\Rb^{2r}))^+ \ar[d]^\cong \\
                R_{n,a}^{an\cdot\vec{1}}(\Rb^{2r})^+ \ar[r] & R_{n,a-1}^{an\cdot\vec{1}}(\Rb^{2r})^+ \ar[r] & (\Poly_n^{n\cdot\vec{1}}(\Rb^{2r})\times\Sym^{a-1}(\Rb^{2r}))^+
            }
        \end{equation*}
        The left vertical map is an $a$-fold branched cover, so on the top degree of compactly supported cohomology, the map it induces is multiplication by $a$. In particular, this gives an isomorphism in rational cohomology, and by the Five Lemma applied to the map of long exact sequences, we see that the cohomology of $R_{n,a-1}^{an\cdot\vec{1}}(\Rb^{2r})$ is as claimed.

        Finally, to see that \eqref{stabF} is an isomorphism for $\bd+1_i=an\cdot\vec{1}$ and $k=a-1$, we apply the Five Lemma to the map of long exact sequences induced by the continuous map of cofiber sequences
        \begin{equation*}
            \xymatrix{
                R_{n,a}^{an\cdot\vec{1}}(\Rb^{2r})^+ \ar[r] \ar[d] & R_{n,a-1}^{an\cdot\vec{1}}(\Rb^{2r})^+ \ar[r] \ar[d] & (\Poly_n^{n\cdot\vec{1}}(\Rb^{2r})\times\Rb^{2r(a-1)})^+ \ar[d]\\
                \ast \ar[r] & (R_{n,a-1}^{an\cdot\vec{1}-1_i}(\Rb^{2r})\times\Rb^{2r})^+ \ar[r] & (\Poly_n^{n\cdot\vec{1}-1_i}(\Rb^{2r})\times\Rb^{2r(a-1)}\times\Rb^{2r})^+
            }
        \end{equation*}
        Using that the map \eqref{stabF} is now an isomorphism for $\bd+1_i=an\cdot\vec{1}$ and $k=a-1$, the downward induction on $k$ now proceeds exactly as above, and this completes the proof of the proposition.\qed

\section{Completing the proof of Theorem \ref{theorem:coho}}
\label{section:final}

In this section we complete the proof of Theorem \ref{theorem:coho}.  We will deduce the theorem in a number of steps, using Theorem \ref{thm:EL}, Theorem \ref{theorem:leray}, and Theorem \ref{theorem:local}. By transfer,
\begin{align*}
  E_2^{p,q}(X,\bd,n)&\cong (\tilde{E}_2^{p,q}(X,D,n)\otimes\Q)^{S_D}\\
    &\cong\left( \bigoplus_{I\in\Pi^D_n} H^p(X_I;\epsilon_I(q)\otimes\Qb)\right)^{S_D} \tag{by Theorem \ref{theorem:leray}}\\
    &\cong\left( \bigoplus_{I\in\Pi^D_n} H^p(X_I;\epsilon_I(q)\otimes\Qb)^{\stab_I}\right)^{S_D}
\end{align*}
Moreover,
\begin{align*}
    H^p(X_I;\epsilon_I(q))^{\stab_I}&\cong (H^p(X_I;\epsilon_I(q)\otimes\Qb)^{S_{I_1}\times\cdots\times S_{I_k}})^{\stab_I/(S_{I_1}\times\cdots\times S_{I_k})}\intertext{and, because $S_{I_1}\times\cdots\times S_{I_k}$ acts trivially on $X_I$,}
    &\cong (H^p(X_I;(\epsilon_I(q)\otimes\Qb)^{S_{I_1}\times\cdots\times S_{I_k}}))^{\stab_I/(S_{I_1}\times\cdots\times S_{I_k})}.
\end{align*}

Our first task after this reduction of the problem is to describe the coefficients $(\epsilon_I(q)\otimes\Qb)^{S_{I_1}\times\cdots\times S_{I_k}}))$.
Before stating the next lemma we need some terminology.  Let $\dim_{\Rb}(X)=N$. Let $\Qb[j]$ denote the rank 1 graded vector space of bidegree $(0,j)$.  Note here that the external tensor product is bigraded.

\begin{lemma}[{\bf Invariants of the coefficient sheaves \boldmath$\epsilon_I(q)$}]
\label{L:SIinv}
Endow $\epsilon_I(q)$ with the bigrading $(0,q)$.
\begin{enumerate}
    \item Let $\dim_{\Rb}(X)=2r+1, r>0$. There is an $S_D$-equivariant isomorphism of bigraded sheaves on $X^D$ :
        \begin{align*}
            (\epsilon_{\hat{0}}(0)\otimes\Qb)&\cong \Qb[0]
        \end{align*}
        where $S_D$ acts trivially on $\Qb[0]$. Further, for all $I\neq\hat{0}\in\Pi^D_n$ :
        \begin{align*}
            (\epsilon_I(q)\otimes\Qb)^{S_{I_1}\times S_{I_2} \times \cdots}&\cong 0.
        \end{align*}
    \item Let $\dim_{\Rb}(X)=2r, r>0$.
        \begin{enumerate}
            \item If $I$ consists only of singletons and $q/(2r(mn-1)-1)$ blocks of size $mn$, then
                \begin{equation}\label{eq:stab6}
                    \stab_I=((\prod_{i=1}^{\ell}S_{I_i})\rtimes S_{q/(2r(mn-1)-1)})\times\prod_{i=1}^m S_{s(I,i)}
                \end{equation}
                where $s(I,i)$ denotes the number of singletons of color $i$, and there is a $\stab_I/(S_{I_1}\times S_{I_2} \times \cdots)$-equivariant isomorphism of bigraded sheaves on $X_I$ :
                \begin{align*}
                    (\epsilon_I(q)\otimes\Qb)^{S_{I_1}\times S_{I_2} \times \cdots}&\cong \Qb[q]
                \end{align*}
                where $\stab_I/(S_{I_1}\times S_{I_2} \times \cdots)\cong S_{q/(2r(mn-1)-1)}\times\prod_{i=1}^m S_{s(I,i)}$ acts on the sheaf $\Qb[q]$ via the alternating representation for $S_{q/(2r(mn-1)-1)}$ and the trivial representation for $\prod_{i=1}^m S_{s(I,i)}$.
            \item For all other $I$:
                \begin{align*}
                    (\epsilon_I(q)\otimes\Qb)^{S_{I_1}\times S_{I_2} \times \cdots}&\cong 0.
                \end{align*}
        \end{enumerate}
\end{enumerate}
\end{lemma}

\begin{proof}[Proof of Lemma \ref{L:SIinv}]
Let $N=\dim_\Rb(X)$.   We first prove the lemma in the case where $|I|=1$, i.e. $I=\hat{1}$ is the terminal object in $\Pi^D_n$ and $\stab_I=S_D$. Note that $S_D$ acts trivially on $X_{\hat{1}}$. By the Goresky-MacPherson Formula (Theorem \ref{theorem:GM}) and the definition of $\epsilon_{K}(q)$, for all $x\in X_{\hat{1}}$, there is an $S_D$-equivariant isomorphism
\[\left(\bigoplus_{K\in\Pi_n^D}\epsilon_{K}(q)\otimes\Qb\right)_x\cong H^q(\widetilde{\Poly}_n^{D}(\Rb^{n});\Qb).\]

Recall the following three facts :
\begin{enumerate}
\item By Theorem~\ref{theorem:local}, for $r>0$ we have :
    \begin{equation}
        H^q(\Poly_n^{\bd}(\Rb^{2r});\Qb)\cong
            \left\{
                \begin{array}{ll}
                    \Qb & q=2r(mn-1)-1\\
                    \Qb & q=0\\
                    0 & \text{else}
                \end{array}\right.
    \end{equation}
   and

       \begin{equation}
        H^q(\Poly_n^{\bd}(\Rb^{2r+1});\Qb)\cong
            \left\{
                \begin{array}{ll}
                                   \Qb & q=0\\
                    0 & \text{else}
                     \end{array}\right.
                     \end{equation}
    \item Transfer and the Goresky-MacPherson formula (Theorem \ref{theorem:GM})
    gives,  for each $q\geq 0$ and all $N\geq 2$:
        \begin{align*}
            H^q(\Poly_n^{\bd}(\Rb^{N});\Qb)&\cong H^q(\widetilde{\Poly}_n^{D}(\Rb^{N});\Qb)^{S_D}\\
            &\cong \left( \bigoplus_{I\in  \Pi^{D}_n}   \tilde{H}_{\cd(I,\Rb^N)-q-2}(\Delta(\overline{\Pi_n^D(\leq I)});\Q)\otimes\coor((\Rb^N_I)\right)^{S_{D}}
        \end{align*}
        \item Theorem \ref{thm:EL} gives that $\Pi_n^D$ satisfies the hypothesis of
        Lemma~\ref{L:BW}.  This lemma then gives that $\dim  \tilde{H}_{\cd(\hat{1},\Rb^N)-q-2}(\Delta(\overline{\Pi_n^D});\Q)    $
            is given by the number of falling chains of $\Pi_n^D$ of length $\cd(\hat{1},\Rb^N)-q$.
\end{enumerate}

Recall that  $\epsilon_{\hat{1}}(q)\otimes \Qb$ is the constant sheaf $\tilde{H}_{\cd(\hat{1},\Rb^N)-q-2}(\Delta(\overline{\Pi_n^D});\Q)\otimes\coor(X_{\hat{1}})$ on $X_{\hat{1}}$.

First suppose that $N=2r+1,r>0$.  Then $S_D$ acts on the sheaves $\coor(X_I)$ by the sign representation for all $I\neq \hat{0}$. The three facts above combine to show that the $S_D$-invariants $(\epsilon_{\hat{1}}(q)\otimes\Qb)^{S_D}$ vanish unless $q=0$, and there exist falling chains $C$ in $\Pi^D_n$ with $\ell(C)=\cd(\hat{1},\Rb^{2r+1})$.  Since $\cd(\hat{1},\Rb^{2r+1})=(2r+1)(|D|-|\hat{1}|)=(2r+1)(|D|-1)$, these conditions are equivalent to:
\begin{align}
\ell(C)&=(2r+1)(|D|-1).\label{eq:lc0}
\end{align}

Now suppose $N=2r,r>0$. Then $S_D$ acts on the sheaf $\coor(X_I)$ by the trivial representation, and the orientation of $X$ induces an $S_D$-equivariant trivialization $\coor(X_I)\cong \Z$ for all $I$. The three facts above combine to show that the $S_D$-invariants $(\epsilon_{\hat{1}}(q)\otimes\Qb)^{S_D}$ vanish unless $q=2r(mn-1)-1$ or $q=0$, and there exist falling chains $C$ in $\Pi^D_n$ with $\ell(C)=\cd(\hat{1},\Rb^N)-(2r(mn-1)-1)$ or $\ell(C)=\cd(\hat{1},\Rb^N)$.
Since $\cd(\hat{1},\Rb^N)=2r(|D|-|\hat{1}|)=2r(|D|-1)$, these conditions are equivalent to:
\begin{align}
\ell(C)&=2r(|D|-mn)+1,\label{eq:lc1}\\
\text{resp.~}\ell(C)&=2r(|D|-1)\label{eq:lc2}.
\end{align}
We now claim that, for any $N\geq 2$, unless $|D(i)|=n$ for $i=1,\ldots,m$, respectively $|D|=1$, there does not exist any falling chain $C$
satisfying \eqref{eq:lc1}, respectively \eqref{eq:lc0} or \eqref{eq:lc2}. Note that we are still assuming $|I|=1$ here. To see the claim, note that if $|D(i)|\ge n$ for all $i$, the longest falling chain $C'$ must consist of one creation of a non-singleton block, followed by singleton mergers.  Since there are $|D|-mn$ singletons left after the first move, it follows that
\[\ell(C')=|D|-mn+1.\]
In particular $\ell(C')\le N(|D|-mn)+1$, with equality as in \eqref{eq:lc1} only when $|D|-mn=0$.  Further, \eqref{eq:lc0} and \eqref{eq:lc2} never occur.  When $|D(i)|=n$ for all $i$, there is one falling chain of length $1$, and we see that $S_D$ must act trivially on it when $N=2r$ since, from Fact 1 above, there is an invariant.

  Similarly, if $|D(i)|<n$ for some $i$, then there are no nontrivial colored $n$-equals partitions, i.e
\begin{equation*}
    \Pi^{D}_n=\{0\}.
\end{equation*}
Thus the unique falling chain has length 0, which is only equal to $N(|D|-1)$ when $|D|=1$.
When $|D|=1$, $\hat{1}=\hat{0}$. Therefore $\coor(X_{\hat{0}}):=H_c(X^D;\Z)\otimes H_c(X^D;\Z)^\vee=\Zb$ with the trivial $S_D$-action regardless of the dimension of $X$. Of course $S_D$ also acts trivially on the unique falling chain. This proves the lemma in the case $|I|=1$.

In the case $|I|>1$, suppose $I$ has non-singleton blocks $I_1,\ldots,I_k$ and singleton blocks $I_{k+1},\ldots,I_\ell$. The projection of $\epsilon_I(q)\otimes\Qb$ onto the
$(S_{I_1}\times S_{I_2} \times \cdots)$-invariants
can be factored as follows: compose the projections $\pi_j$ onto the invariants for the group which fixes $I_j$ setwise and $D\setminus I_j$ pointwise, for $j=1,\ldots,k$.

By Proposition~\ref{proposition:kunneth}, any class in $H_\ast(\Delta(\overline{\Pi^D_n(\le I)});\Z)\otimes \coor(X_I)$ is a product of classes coming from the partitions with only one non-singleton block $I_j$ for $j=1,\ldots,k$.  The argument above shows that the projection $\pi_j$ is 0 unless $N$ is even and $|I_j\cap D(i)|=n$ for $i=1,\ldots,m$.   In the case that $N=2r$ is even, if all $I_j$ for $j=1,\ldots, k$ have $|I_j\cap D(i)|=n$ for $i=1,\ldots,m$, we have a single dimension of $(S_{I_1}\times S_{I_2} \times \cdots)$--invariants for any $q$ divisible by $2r(mn-1)-1$ and no invariants for any other $q$.  This gives the Statement 1 and the first part of Statement 2 of the lemma.

For the second part of Statement 2, the group acting here is non-canonically isomorphic to
\[((S_{n}^{\times m})\wr S_{q/(2r(mn-1)-1)})\times \prod_{i=1}^m S_{s(I,i)}.\]

Under the isomorphism of Proposition~\ref{proposition:kunneth}, the isomorphism of Lemma~\ref{L:SIinv} takes the form
\begin{equation}\label{eq:action}
    (\epsilon_I(q)\otimes\Qb)^{S_{I_1}\times\cdots\times S_{I_\ell}}\cong \bigboxtimes_{j=1}^\ell \Qb\langle I_j\rangle
\end{equation}
where $\Qb\langle I_j\rangle$ denotes, for $|I_j|=mn$, the constant sheaf on $X_{I_j}$ with stalk the rank 1 graded vector space of bidegree $(0,2r(mn-1)-1)$ corresponding to the unique falling chain in $\Pi^{I_j}_n$ of length $1$, and where $\Qb\langle I_j\rangle$ denotes, for $|I_j|=1$, the constant sheaf on $X_{I_j}$ with stalk the rank 1 graded vector space of bidegree (0,0).  Passing to the quotient $\stab_I/(S_{I_1}\times\cdots S_{I_\ell})\cong S_{q/(2r(mn-1)-1)}\times\prod_{i=1}^m S_{s(I,i)}$, we see that the $S_{q/(2r(mn-1)-1)}$ acts on the right-hand side of \eqref{eq:action} according to the K\"unneth isomorphism and the graded rule of signs, i.e. by permuting classes of odd total degree past each other via the sign representation, while $S_{s(I,i)}$ acts by permuting classes of total degree 0 past each other, i.e. via the trivial representation.
\end{proof}

Back to the proof of Theorem \ref{theorem:coho}.   First note that Theorem~\ref{theorem:leray} gives
\begin{align*}
(\widetilde{E}_2^{p,q}(X,D,n)\otimes\Qb)^{S_D}&\cong \left( \bigoplus_{I\in\Pi_n^D}H^p(X_I;\epsilon_I(q)\otimes\Qb)\right)^{S_D}\\
&\cong \left( \bigoplus_{I\in\Pi_n^D}H^p(X_I;\epsilon_I(q)\otimes\Qb)^{S_{I_1}\times S_{I_2} \times\cdots }\right)^{S_D}\\
\end{align*}
where the second isomorphism follows from basic linear algebra.  Since $S_{I_1}\times S_{I_2} \times\cdots $ acts trivially on $X_I$,   Lemma~\ref{L:SIinv} gives that:
\begin{itemize}
\item If $\dim_\Rb(X)=2r+1,~r>0$ then  $(\epsilon_I(q)\otimes\Qb)^{S_{I_1}\times S_{I_2} \times\cdots }=0$ unless $I$ consists only of singletons and $q=0$; and
\item if $\dim_\Rb(X)=2r,~r>0$ then $(\epsilon_I(q)\otimes\Qb)^{S_{I_1}\times S_{I_2} \times\cdots }=0$ unless $I$ consists only of singletons and $q/(2r(mn-1)-1)$ blocks of size $mn$.
\end{itemize}

We conclude that if $\dim_\Rb(X)=2r+1,~r>0$ then

\begin{equation*}
(\widetilde{E}_2^{p,q}(X,D,n)\otimes\Qb)^{S_D} \cong \left\{
\begin{array}{ll}
H^p(X^D;\Q)^{S_D}&q=0\\
0&q>0
\end{array}\right.
\end{equation*}
This proves the first statement of the theorem.

For the second statement, if $\dim_\Rb(X)=2r,~r>0$ then  the above gives that

\begin{equation}
\label{equation:sfinvariance}
(\widetilde{E}_2^{p,q}(X,D,n)\otimes\Qb)^{S_D} \cong \left( \bigoplus_{\substack{I\in  \Pi^D_n\\I= \text{\rm singletons and blocks of size $mn$}\\
 \text{\rm with $q/(2r(mn-1)-1)$ blocks of size $mn$}  }}
H^p(X_I; \epsilon_I(q)\otimes\Qb ) \right)^{S_D}.
\end{equation}

We are now in a position to prove the second statement of Theorem \ref{theorem:coho}.
Let $J\in \Pi^D_n$ be a partition composed of singletons and $q/(2r(mn-1)-1)$ blocks of size $mn$.
Then the $S_D$-representation
$$
\bigoplus_{\substack{I\in  \Pi^D_n\\I \ \text{\rm singletons and blocks of size $mn$}\\
 \text{\rm with $q/(2r(mn-1)-1)$ blocks of size $mn$}  }}
H^p(X_I; \epsilon_I(q)\otimes\Qb )
$$
is the induction from $\stab_{J}$ up to $S_D$ of
$$
H^p(X_{J}; \epsilon_{J}(q)\otimes\Qb ).
$$
Thus, by Frobenius reciprocity,
$$
\left( \bigoplus_{\substack{I\in  \Pi^D_n\\I= \text{\rm singletons and blocks of size $mn$}\\
 \text{\rm with $q/(2r(mn-1)-1)$ blocks}  }}
H^p(X_I; \epsilon_I(q)\otimes\Qb ) \right)^{S_D}
=H^p(X_{J}; \epsilon_{J}(q)\otimes\Qb )^{\stab_{J}}.
$$
We have
$$
H^p(X_{J}; \epsilon_{J}(q)\otimes\Qb )^{\stab_{J}} =\left( H^p(X_{J}; \epsilon_{J}(q)\otimes\Qb )^{S_{J_1}\times S_{J_2} \times \cdots}  \right)^{\stab_{J}/(S_{J_1}\times S_{J_2} \times \cdots)}.
$$
Note that, $S_{J_1}\times S_{J_2} \times \cdots$ acts trivially on $X_J$.  By Lemma~\ref{L:SIinv}, we have that
$(\epsilon_{J}(q)\otimes\Qb)^{S_{J_1}\times S_{J_2} \times \cdots}\cong\Qb[q]$ where $\Qb[q]$ denotes the constant rank 1 graded sheaf on $X_{J}$ in bidegree $(0,q)$.  By Lemma \ref{L:SIinv}, $\stab_{J}/(S_{J_1}\times S_{J_2} \times \cdots)$ acts on $\Qb[q]$ by the sign representation for permutations of nonsingleton blocks and the trivial representation for permutations of singletons.

Let $J$ have $k$ blocks of size $mn$ and $s(J,i)$ singletons of color $i$.  Note that $k\le |D(i)|/n$ for all $i=1,\ldots,m$. Then $X_J\cong X^{k} \times \prod_{i=1} ^m X^{s(J,i)}$ and, by the definition of $\epsilon_J(q)$ and Lemma \ref{L:SIinv},
$H^p(X_{J};\epsilon_J(q)\otimes\Qb^{S_{J_1\times\cdots}} )$ is the degree $(p,q)$ part of
\begin{equation*}
    H^*(X;\Qb[2r(mn-1)-1])^{\tensor k} \tensor \bigotimes_{i=1}^m H^*(X; \Qb[0])^{\tensor s(J,i)},
\end{equation*}
where the cohomological degree contributes only to the $p$ degree, and where $\stab_{J}/(S_{J_1}\times S_{J_2} \times \cdots)\simeq S_k \times S_{\ell_1} \times S_{\ell_2} \times \cdots$ acts in the usual (graded) way from the K\"unneth formula. Thus $H^p(X_{J}; \epsilon_{J}(q)\otimes\Qb )^{\stab_{J}} $ is the degree $(p,q)$ part of
\begin{equation*}
    \Sym_{gr}^k H^*(X; \Qb[2r(mn-1)-1])\tensor  \bigotimes_{i=1}^m \Sym_{gr}^{s(J,i)} H^*(X; \Qb[0])
\end{equation*}
as claimed.

When $X$ is a smooth complex variety,  Theorem \ref{theorem:leray} applied to Equation \eqref{equation:sfinvariance} gives the weights as claimed in the third statement of the theorem.
\qed

\subsection*{Acknowledgements}
We thank A. Berglund, J. Grodal, M. Kapranov, B. Knudsen, J.P. May, J. Miller, M. Nori and C. Westerland for helpful comments and questions.  We are very appreciative to Tom Church for making extensive comments on an earlier draft of this paper, and to Dan Petersen for pointing out a crucial example which suggested a much richer picture.  Finally, we thank the anonymous referees for many comments that helped to improve the paper.

The first author was supported in part by National Science Foundation Grant Nos. DMS-1105643, 
DMS-1406209, and the Jump Trading Mathlab Research Fund. The second author was supported in part by National Science Foundation Grant No. DMS-1400349.  The third author was supported by an American Institute of Mathematics Five-Year Fellowship, a Packard Fellowship for Science and Engineering, a Sloan Research Fellowship, and National Science Foundation grants DMS-1301690 and DMS-1652116.

\bigskip{\noindent
Dept. of Mathematics, University of Chicago\\
E-mail: farb@math.uchicago.edu\\
\\
Dept. of Mathematics, University of California, Irvine\\
E-mail: wolfson@uci.edu\\
\\
Dept. of Mathematics, University of Wisconsin--Madison; American Institute of Mathematics\\
E-mail: mmwood@math.wisc.edu


\small

\begin{thebibliography}{DGM00}

\bibitem[AF59]{AF}
A.~Andreotti and T.~Frankel.
\newblock The Lefschetz theorem on hyperplane sections.
\newblock {\em Ann. of Math.} (2), 69: 713-–717, 1959.

\bibitem[Ara16]{Arabia}
A.~Arabia.
\newblock Espaces de configuration g\'{e}n\'{e}ralis\'{e}s. Espaces topologiques $i$-acycliques. Suites spectrales ``basiques.''
\newblock arXiv:1609.00522.

\bibitem[Arn69]{Ar}
V.~Arnol'd.
\newblock The cohomology ring of the colored braids group.
\newblock {\em Mat. Zametki}, 5:227--231, 1969.

\bibitem[Arn70]{Ar70b}
V.~Arnol'd.
\newblock On some topological invariants of algebraic functions.
\newblock {\em Tr. Mosc. Mat. Obsc.}, pages 27--46, 1970.

\bibitem[BCT89]{BCT}
C.-F. Bodigheimer, F.~Cohen, and L.~Taylor.
\newblock On the homology of configuration spaces.
\newblock {\em Topology}, 28(1):111--123, 1989.

\bibitem[BE97]{BE}
A.~Bj\"orner and T.~Ekedahl.
\newblock Subspace arrangements over finite fields: cohomological and
  enumerative aspects.
\newblock {\em Adv. Math.}, 129(2):159--187, 1997.

\bibitem[Ben76]{Benkoski}
S.-J. Benkoski.
\newblock The probability that {$k$} positive integers are relatively
  {$r$}-prime.
\newblock {\em J. Number Theory}, 8(2):218--223, 1976.

\bibitem[BM14]{BM}
M.~Bendersky and J.~Miller.
\newblock Localization and homological stability of configuration spaces.
\newblock {\em Q. J. Math.}, 65(3):807--815, 2014.

\bibitem[BT82]{BottTu}
R.~Bott and L.~Tu.
\newblock {\em Differential Forms in Algebraic Topology}, volume~82 of {\em
  GTM}.
\newblock Springer--Verlag, 1982.

\bibitem[BW96]{BW}
A.~Bj\"orner and M.~Wachs.
\newblock Shellable nonpure complexes and posets. {I}.
\newblock {\em Trans. Amer. Math. Soc.}, 348(4):1299--1327, 1996.

\bibitem[BW97]{BW2}
A.~Bj\"orner and M.~Wachs.
\newblock Shellable nonpure complexes and posets. {II}.
\newblock {\em Trans. Amer. Math. Soc.}, 349(10):3945--3975, 1997.

\bibitem[Chu12]{Ch}
T.~Church.
\newblock Homological stability for configuration spaces of manifolds.
\newblock {\em Invent. Math.}, 188:465--504, 2012.

\bibitem[CCMM91]{CCMM}
F.~Cohen, R.~Cohen, B.~Mann, R.J.~Milgram.
\newblock The topology of rational functions and divisors of surfaces.
\newblock {\em Acta Math.}, 166(3-4):162--221, 1991.


\bibitem[Del71]{De}
P.~Deligne.
\newblock Theorie de {H}odge {II}.
\newblock {\em Inst. Hautes \'{E}tudes Sci. Publ. Math.}, (40):5--57, 1971.

\bibitem[DGM00]{DGM}
P.~Deligne, M.~Goresky, and R.~MacPherson.
\newblock L'alg\`ebre de cohomologie du compl\'ement, dans un espace affine,
  d'une famille finie de sous-espaces affines.
\newblock {\em Michigan Math. J.}, 48:121--136, 2000.



\bibitem[dLS01]{LS}
M.~de~Longueville and C.~Schultz.
\newblock The cohomology rings of complements of subspace arrangements.
\newblock {\em Math. Ann.}, 319(4):625--646, 2001.

\bibitem[FT05]{FTa}
Y.~F\'elix and D.~Tanr\'e.
\newblock The cohomology algebra of unordered configuration spaces.
\newblock {\em J. London Math. Soc.}, 72(2):525--544, 2005.

\bibitem[FW16]{FW}
B.~Farb and J.~Wolfson.
\newblock Topology and arithmetic of resultants, {I}.
\newblock {\em New York J. Math.}, 22:801--821, 2016. Also see the Corrigendum, Volume 25 (2019), 195-197.

\bibitem[Gad]{Ga}
N.~Gadish.
\newblock Representation Stability For Families Of Linear
Subspace Arrangements.
\newblock {\em Adv. Math.}, to appear.

\bibitem[Geg85]{Gegenbauer}
L.~Gegenbauer.
\newblock Asymptotische {Gesetze} der {Zahlentheorie}.
\newblock {\em Denkschriften der Kaiserlichen Akademie der Wissenschaften,
  Mathematisch-Naturwissenschaftliche Classe}, 49:37--80, 1885.

\bibitem[Get96]{Getzler}
E.~Getzler
\newblock Mixed Hodge structures of configuration spaces.
\newblock arXiv:alg-geom/9510018.

\bibitem[Ho]{Ho}
Q. Ho, Densities and stability via factorization homology, preprint, Feb.\ 2018.

\bibitem[KM]{KM}
A.~Kupers and J.~Miller.
\newblock ${E}_n$-cell attachments and a local-to-global principle for
  homological stability.
\newblock arXiv:1405.7087.

\bibitem[Knu]{Kn}
B.~Knudsen.
\newblock Betti numbers and stability for configuration spaces via
  factorization homology.
\newblock arXiv:1405.6696.

\bibitem[MD]{Morrison}
K.~Morrison and Z.~Dong.
\newblock The probability that random polynomials are relatively r-prime.
\newblock http://www.calpoly.edu/~kmorriso/Research/RPFF04-2.pdf.

\bibitem[Mer74]{Mertens}
F.~Mertens.
\newblock Ueber einige asymptotische {G}esetze der {Z}ahlentheorie.
\newblock {\em J. Reine Angew. Math.}, 77:289--338, 1874.

\bibitem[MS]{MS}
J. Milnor and J. Stasheff, Characteristic classes,  {\em Annals of Math. Studies}, No. 76, Princeton Univ.\ Press, Princeton, 1974.

\bibitem[Pet]{Pe}
D.~Petersen.
\newblock A spectral sequence for stratified spaces and configuration spaces of
  points.
\newblock arXiv:1603.01137.

\bibitem[Pet2]{Pet2}
D.~Petersen.
\newblock Cohomology of generalized configuration spaces.
\newblock arXiv:1807.07293.

\bibitem[PS08]{PS}
C.~Peters and J.~Steenbrink.
\newblock {\em Mixed {H}odge {S}tructures}.
\newblock Springer--Verlag, 2008.

\bibitem[RW13]{RW}
O.~Randal-Williams.
\newblock Homological stability for unordered configuration spaces.
\newblock {\em Q. J. Math.}, 64(1):303--326, 2013.

\bibitem[Sai90]{Sa}
M.~Saito.
\newblock Mixed {H}odge modules.
\newblock {\em Publ. Res. Inst. Math. Sci.}, 26:221--333, 1990.

\bibitem[Seg79]{Se}
G.~Segal.
\newblock The topology of spaces of rational functions.
\newblock {\em Acta Math.}, 143(1-2):39--72, 1979.

\bibitem[Tot96]{To}
B.~Totaro.
\newblock Configuration spaces of algebraic varieties.
\newblock {\em Topology}, 35(4):1057--1067, 1996.

\bibitem[VW15]{VakilWood}
R.~Vakil and M.~M.~Wood
\newblock Discriminants in the Grothendieck ring.
\newblock {\em Duke Math. J.}, 164(6):1139--1185, 2015.

\bibitem[Vas92]{Va}
V.A. Vassiliev.
\newblock {\em Complements of discriminants of smooth maps: topology and
  applications}, volume~98.
\newblock American Mathematical Society, 1992.

\bibitem[Wal88]{Wa}
J.~Walker.
\newblock Canonical homeomorphisms of posets.
\newblock {\em Europ. J. Combinatorics}, 9:97--107, 1988.

\end{thebibliography}
\end{document}